\documentclass[a4,10pt]{article}

\usepackage{amsthm}
\usepackage[percent]{overpic}
\usepackage[pdftex,colorlinks=true]{hyperref}
\usepackage{authblk}
\usepackage{chngcntr}
\usepackage{apptools}
\usepackage[margin=1.2in]{geometry}
\AtAppendix{\counterwithin{lem}{section}}

\usepackage{trimclip}
\usepackage[utf8]{inputenc}
\usepackage{bbm}
\usepackage{subcaption}
\usepackage{amssymb,amsmath}
\usepackage{bm}
\usepackage{enumitem}
\usepackage{tabularx}
\usepackage{mathtools}
\setenumerate{label=\alph*)}
\usepackage{hyperref}
\newcommand{\qta}{\quad\text{ and }\quad}

\newcommand{\fnu}{\mathbf{F}^{[\nu]}}

\newcommand{\tend}{t_\mathrm{end}}
\def\code#1{\texttt{#1}}
\usepackage{xspace}
\newcommand{\MATLAB}{\textsc{Matlab}\xspace}
\renewcommand{\vec}[1]{\begin{psmallmatrix}#1\end{psmallmatrix}}
\usepackage{graphicx}
\makeatletter
\renewcommand*\dashline{\rotatebox[origin=c]{90}{$\dabar@\dabar@\dabar@$}}
\makeatother
\usepackage{trimclip}

\usepackage{xcolor}
\hypersetup{
	colorlinks,
	linkcolor={red!50!black},
	citecolor={blue!50!black},
	urlcolor={blue!80!black}
}

\usepackage{forest}
\forestset{
	declare toks={elo}{}, 
	anchors/.style={grow'=90, anchor=#1,child anchor=#1,parent anchor=#1},
	dot/.style={tikz+={\fill (.child anchor) circle[radius=#1];}},
	dot/.default=2pt,
	decision edge label/.style n args=3{
		edge label/.expanded={node[midway,auto=#1,anchor=#2,\forestoption{elo}]{\strut$\unexpanded{#3}$}}
	},
	decision/.style={if n=1
		{decision edge label={left}{east}{#1}}
		{decision edge label={right}{west}{#1}}
	},
	decision tree/.style={
		for tree={grow'=90,
			s sep=2.5pt, l=0pt, l sep =0.5pt, outer sep =-1.5pt,
			if n children=0{anchors=west}{
				if n=1{anchors=west}{anchors=west}},
			math content,
		},
		anchors=west, outer sep=-1.5pt,
		dot=2pt, for descendants=dot,
		delay={for descendants={split option={content}{;}{content,decision}}},
	},
	rooted tree/.style={
		for tree={
			grow'=90,
			parent anchor=center,
			child anchor=center,
			s sep=2.5pt,
			l sep =1pt,
			if level=0{
				baseline
			}{},
			delay={
				if content={*}{
					content=,
					append={[]}
				}{}
			}
		},
		before typesetting nodes={
			for tree={
				circle,
				fill,
				minimum width=3pt,
				inner sep=0pt,
				child anchor=center,
			},
		},
		before computing xy={
			for tree={
				l=5pt,
			}
		}
	}
}

\DeclareDocumentCommand\ct{o}{\Forest{decision tree [#1]}}
\DeclareDocumentCommand\rt{o}{\Forest{rooted tree [#1]}}
\newcommand{\et}[2]{\resizebox{#1 cm}{!}{\marginbox*{0pt -0.45\height pt 0pt 0pt}{#2}}}
\renewcommand{\tt}[2]{\resizebox{#1 cm}{!}{\marginbox*{0pt -0.25\height pt 0pt 0pt}{#2}}}
\DeclareMathOperator{\NB}{NB}

\newcommand{\bd}{\bm\delta}
\newcommand{\N}{\mathbb{N}}
\newcommand{\R}{\mathbb{R}}

\newcommand{\dt}{h}

\newcommand{\y}{\mathbf{y}}
\newcommand{\Y}{\mathbf{Y}}
\newcommand{\byi}{\y^{(i)}}
\newcommand{\byj}{\y^{(j)}}
\newcommand{\yi}{y^{(i)}}

\newcommand{\D}{\mathbf{D}}
\newcommand{\F}{\mathbf{F}}

\newcommand{\Fnu}{\mathbf{F}^{[\nu]}}
\newcommand{\Fmu}{\mathbf{F}^{[\mu]}}
\newcommand{\dF}{\mathcal{F}}

\newcommand{\f}{\mathbf{f}}
\newcommand{\A}{\mathbf{A}}
\renewcommand{\b}{\mathbf{b}}
\renewcommand{\c}{\mathbf{c}}

\renewcommand{\O}{{\mathcal O}}

\theoremstyle{definition}
\newtheorem{theorem}{Theorem}

\newtheorem{lemma}[theorem]{Lemma}

\newtheorem{rem}[theorem]{Remark}
\newtheorem{corollary}[theorem]{Corollary}
\newtheorem{example}[theorem]{Example}

\begin{document}
	\title{Order conditions for Runge--Kutta-like methods with solution-dependent coefficients}
	
	\author[1]{Thomas Izgin} 
\author[2]{David I.\ Ketcheson}
\author[1]{Andreas Meister}
\affil[1]{Department of Mathematics and Natural Sciences, University of Kassel, Germany}
\affil[2]{CEMSE Division, King Abdullah University of Science \& Technology (KAUST), Saudi Arabia}
\affil[1]{izgin@mathematik.uni-kassel.de\ \&\ meister@mathematik.uni-kassel.de}
\affil[2]{david.ketcheson@kaust.edu.sa}
\setcounter{Maxaffil}{0}
\renewcommand\Affilfont{\itshape\small}
\maketitle
\begin{abstract}
	In recent years, many positivity-preserving schemes for initial value problems have been constructed by modifying a Runge--Kutta (RK) method by weighting the right-hand side of the system of differential equations with solution-dependent factors.  These include the classes of modified Patankar--Runge--Kutta (MPRK) and Geometric Conservative (GeCo) methods.  Compared to traditional RK methods, the analysis of accuracy and stability of these methods is more complicated.
	In this work, we provide a comprehensive and unifying theory of order conditions for such RK-like methods, which differ from original RK schemes in that their coefficients are solution-dependent. The resulting order conditions are themselves solution-dependent and obtained using the theory of NB-series, and thus, can easily be read off from labeled N-trees. 
	We present for the first time order conditions for MPRK and GeCo schemes of arbitrary order; For MPRK schemes, the order conditions are given implicitly in terms of the stages. From these results, we recover as particular cases all known order conditions from the literature for first- and second-order GeCo as well as first-, second- and third-order MPRK methods. Additionally, we derive sufficient and necessary conditions in an explicit form for 3rd and 4th order GeCo schemes as well as 4th order MPRK methods. We also present a new 4th order MPRK method within this framework and numerically confirm its convergence rate.
\end{abstract}

\section{Introduction}
In the numerical solution of differential equations, it is often difficult to ensure the positivity of quantities such as density, energy, or concentrations,
especially when high-order discretizations are used.
In recent years, certain classes of numerical methods have been designed to preserve positivity in the solution of initial value problems by modifying the terms that could otherwise make the solution negative while preserving an overall desired order of accuracy 
\cite{gBBKS,MPRK2,MPRK3,MPDeC,martiradonna2020geco, SIRK23,SSPMPRK2,SSPMPRK3}. 
Such methods include the modified Patankar--Runge--Kutta (MPRK)
and Geometric Conservative (GeCo) schemes.  The
derivation of order conditions for these schemes is challenging.
Besides the so-called MPDeC methods \cite{MPDeC}, which are Patankar-type schemes based on arbitrary high order deferred correction methods, order conditions only up to order at most three were derived so far.  The technicality and length of the corresponding proofs increases with the order of the method.  
The motivation for the present work is the need for a more systematic derivation
of order conditions for these methods.

The key idea in our approach is to write these methods
in the form of a Runge--Kutta (RK) method but with 
solution-dependent coefficients, by incorporating the positivity-preserving
factors into the coefficients of the method.  We refer to the
resulting class of methods as \emph{non-standard Runge--Kutta} (NSRK)
methods.
Writing the methods in this form enables the application of the extensive existing theory of order conditions for RK methods.
The task of writing an MPRK method as an NSRK method is complicated
by the fact that at each Runge--Kutta stage, each term in the 
differential equation may be multiplied by a different
positivity-preserving factor.  Thus the NSRK method must accommodate
different solution-dependent coefficients at each stage and for each term.  Therefore, even though
MPRK methods themselves are not generally viewed as additive methods,
we are led to study order conditions for additive RK methods with
solution-dependent coefficients, which we refer to as NSARK methods.
In this work we develop a general theory that facilitates
the construction of order conditions not only for MPRK and GeCo methods, but for
any method that can be written as an additive Runge-Kutta method
with solution-dependent coefficients.

The rest of this paper is organized as follows.  In Section \ref{sec:NSARK-OC},
we provide a theoretical basis for the use of non-standard NB-series,
by extending results from \cite{B16} to the case of colored trees
and solution-dependent coefficients.  By comparing these series
with that of the exact solution, we obtain a general form of order conditions.
In Sections \ref{sec:GeCo} and \ref{sec:MPRK}, we apply these general order conditions to the cases of
MPRK and GeCo schemes, determining more specifically the conditions that
must be satisfied to achieve order up to three, verifying
results of \cite{MPRK2,MPRK3,martiradonna2020geco} and providing new conditions for 3rd and 4th order GeCo schemes as well as 4th order MPRK methods. Finally, we even construct a new fourth order MPRK scheme and numerically measure the expected order of convergence.
\section{Order Conditions for NSARK methods} \label{sec:NSARK-OC}

In this section we study the approximate solution of the initial value problem
\begin{align}\label{ivp}
	\y'(t)  =\F(\y(t))= \sum_{\substack{\nu=1}}^N  \Fnu(\y(t)), \quad  \y(0) & = \y^0\in \R^d.
\end{align}
One approach to solve \eqref{ivp} is the additive Runge-Kutta (ARK) method
\begin{subequations}
	\begin{align*}
		\byi & = \y^n + h \sum_{j=1}^s  \sum_{\substack{\nu=1}}^N a^{[\nu]}_{ij}\Fnu(\byj), \quad i=1,\dotsc,s,\\
		\y^{n+1} & = \y^n + h \sum_{j=1}^s \sum_{\substack{\nu=1}}^N  b^{[\nu]}_j\Fnu(\byj).
	\end{align*}
\end{subequations} 
The accuracy of standard RK methods can be understood through the use of trees and B-series, which are formal power series used to represent exact
and approximate solutions of an initial value problem \cite{B16,HWButcherTrees}.
Similarly, ARK methods can be studied using colored trees and NB-Series \cite{Ntrees}.  
In this work, we further extend this theory in order to study methods of the form
\begin{subequations} \label{eq:nsark}
	\begin{align}
		\byi & = \y^n + h \sum_{j=1}^s  \sum_{\substack{\nu=1}}^N a^{[\nu]}_{ij}(\Y^n,h)  \Fnu(\byj), \quad i=1,\dotsc,s,\\
		\y^{n+1} & = \y^n + h \sum_{j=1}^s \sum_{\substack{\nu=1}}^N b^{[\nu]}_j(\Y^n,h) \Fnu(\byj),
	\end{align}
\end{subequations}
where $\Y^n=(\y^n \dashline\y^{(1)}\dashline\dotsc\dashline\y^{(s)}\dashline \y^{n+1})$.
We call schemes of this form \emph{non-standard additive Runge--Kutta (NSARK) methods}; they differ from ARK methods in that their coefficients are allowed to depend on the solution and the step size.
We study the local order of accuracy of \eqref{eq:nsark} through
the use of a generalization of NB-series.

Usually, an ARK method is used in order to apply very
different RK schemes to the different components of $\Fnu$.  In this work, the coefficients of each part of the ARK method will be almost equal, differing only in how they depend on $h$ and $\Y^n$.
Note that the stages $\byi=\byi(\y^n)$ and the update $\y^{n+1}=\y^{n+1}(\y^n)$ may be interpreted as functions of $\y^n$, so that $a^{[\nu]}_{ij}, b^{[\nu]}_j$ may as well be interpreted only as functions of $\y^n$ and the step size $h$.



In order to analyze NSARK methods, we first review the concepts of
N-trees and NB-series.  We then develop extensions of standard results
in order to show that 
the order conditions for NSARK schemes can be written in a way that is structurally the same as those of ARK methods.  That is, they differ only by the replacement of fixed coefficients by solution-dependent coefficients. To prove this, we follow the approach of Butcher \cite{B16}. 

\subsection{Preliminaries}

A \emph{rooted tree} is a cycle-free, connected graph with one node designated as the root \cite{B16}. More precisely, a rooted tree can be understood as the underlying undirected graph of an arborescence, for which the root is the uniquely determined node with no incoming arc \cite{KV2012}.  We consider colored trees, in which each node possesses one of $N$ possible colors.  We denote the set of all colored rooted trees, the so-called \emph{$N$-trees}, by $NT$. We indicate the color $\nu\in \{1,\dotsc, N\}$ of the tree $\rt[]$ by writing $\rt[]^{[\nu]}$. In general, a colored rooted tree $\tau$ with a root color $\nu$ can be written in terms of its colored \emph{children} $\tau_1,\dotsc,\tau_k$ by writing 
\begin{equation}\label{eq:tau}
	\tau=[\tau_1,\dotsc,\tau_k]^{[\nu]}=[\tau_1^{m_1},\dotsc,\tau_r^{m_r}]^{[\nu]},
\end{equation}
where the children $\tau_1,\dotsc,\tau_k$ are the connected components of $\tau$ when the root together with its edges are removed. Moreover, the neighbors of the root of $\tau$ are the roots of the corresponding children.
In the latter representation of $\tau$ in \eqref{eq:tau}, $m_i$ is the number of copies of $\tau_i$ within $\tau_1,\dotsc,\tau_k$, which already includes the fact that we do not distinguish between trees whose children are permuted.

\begin{example}
	For simplicity, we consider only one color in this example, that is $N=1$ and $\rt[]^{[1]}=\rt[]$. The children of the tree $\tau=\rt[[],[[],[]]]$ are given by $\tau_1=\rt[]$ and $\tau_2=\rt[[],[]]$ and the respective roots are the lowest nodes. In terms of the representation \eqref{eq:tau} we can write $\rt[[],[[],[]]]=[\rt[],\rt[[],[]]]=[\rt[],[\rt[],\rt[]]]=[\rt[],[\rt[]^2]]=[[\rt[]^2],\rt[]]$. 
\end{example}

The \emph{order} of a colored tree $\tau$ is denoted by $\lvert\tau\rvert$ and equals the number of its nodes. The \emph{symmetry} $\sigma$ and \emph{densitity} $\gamma$ of $\tau$ from \eqref{eq:tau} are defined by
\begin{equation}\label{eq:sigmagamma}
	\begin{aligned}
		\sigma(\tau)&=\prod_{j=1}^rm_j!\sigma(\tau_j), &&\sigma(\rt[]^{[\nu]})=1,\quad \nu=1,\dotsc,N,\\
		\gamma(\tau)&=\lvert \tau\rvert \prod_{i=1}^k\gamma(\tau_i),&&\gamma(\rt[]^{[\nu]})=1,\quad \nu=1,\dotsc,N.
	\end{aligned}
\end{equation}
Observe that $\sigma$ depends on the coloring of $\tau$, while $\gamma$ does not since already $\lvert \tau \rvert$ is independent of the coloring. For instance we find $\sigma(\left[\rt[]^{[1]},\rt[]^{[2]}\right]^{[3]})=1$ since the children are not identical, while $\sigma(\left[\rt[]^{[1]},\rt[]^{[1]}\right]^{[3]})=2$. Meanwhile, $\gamma(\left[\rt[]^{[1]},\rt[]^{[2]}\right]^{[3]})=\gamma(\left[\rt[]^{[1]},\rt[]^{[1]}\right]^{[3]})=3$.

Next we introduce the set $NT_q$ of all $N$-trees up to order $q$. The sets $NT_q$ for $q=1,2,3$ read
\begin{equation}    
	\begin{aligned}
		NT_1&=\{\rt[]^{[\nu]}\mid \nu=1,\dotsc,N\},\\
		NT_2&=NT_1\cup \left\{\et{0.6}{\ct[\hphantom{.}^{[\mu]}[\hphantom{.}^{[\nu]}]]}\bigg| \nu,\mu=1,\dotsc,N\right\},\\
		NT_3&=NT_2\cup \Biggl\{\et{0.6}{\ct[\hphantom{.}^{[\mu]}[\hphantom{.}^{[\nu]}[\hphantom{.}^{[\xi]}]]]}\Bigg| \nu,\mu,\eta=1,\dotsc,N\Biggr\}\cup \Biggl\{\et{1.2}{\ct[\hphantom{.}^{[\mu]}[\hphantom{.}^{[\nu]}][\hphantom{.}^{[\xi]}]]}\Bigg|\nu,\mu,\eta=1,\dotsc,N\Biggr\},
	\end{aligned}
\end{equation}
where we used the representation $\left[\rt[]^{[\nu]}\right]^{[\mu]}=\tt{0.5}{\ct[\hphantom{.}^{[\mu]}[\hphantom{.}^{[\nu]}]]}$, $\left[\left[\rt[]^{[\xi]}\right]^{[\nu]}\right]^{[\mu]}=\tt{0.5}{\ct[\hphantom{.}^{[\mu]}[\hphantom{.}^{[\nu]}[\hphantom{.}^{[\xi]}]]]}$ and $\left[\rt[]^{[\nu]},\rt[]^{[\xi]}\right]^{[\mu]}=\tt{1}{\ct[\hphantom{.}^{[\mu]}[\hphantom{.}^{[\nu]}][\hphantom{.}^{[\xi]}]]}$. In particular, the root is depicted as the lowest node.

For the following analysis, elementary differentials $\dF\colon NT\to \mathcal C(\R^d,\R^d)$ for colored trees (see \cite{Ntrees}) play a major role and are recursively defined by
\begin{equation}\label{eq:elemdiff}
	\begin{aligned}
		\dF(\rt[]^{[\nu]})(\y)&=\Fnu(\y), \quad \nu=1,\dotsc, N,\\
		\dF([\tau_1,\dotsc,\tau_k]^{[\nu]})(\y)&=\sum_{i_1,\dotsc,i_k=1}^d\partial_{i_1,\dotsc,i_k}\Fnu(\y)\dF_{i_1}(\tau_1)(\y)\cdots \dF_{i_k}(\tau_k)(\y), \quad \nu=1,\dotsc, N.
	\end{aligned}
\end{equation}
An important result in \cite{B16,HWButcherTrees} is the representation of the analytical solution of \eqref{ivp} in terms of an NB-series
\[\NB(u,\y)=\y+ \sum_{\tau\in NT}\frac{h^{\lvert \tau\rvert}}{\sigma(\tau)}u(\tau)\dF(\tau)(\y),\]
where $u\colon NT\to \R$ and $\y\in \R^d$. Note that $\NB(u,\y)$ is defined only if $\Fmu\in \mathcal C^\infty$ for $\mu=1,\dotsc, N$. For $\Fmu\in \mathcal C^{k+1}$, we truncate the NB-series and introduce
\[\NB_k(u,\y)=\y+ \sum_{\tau\in NT_k}\frac{h^{\lvert \tau\rvert}}{\sigma(\tau)}u(\tau)\dF(\tau)(\y),\]
and point out that $\NB_0(u,\y)=\y.$ With that, we can formulate a theorem concerning the NB-series expansion of the solution to the differential equation at some time $t+h$.
\begin{theorem}[{\cite[Theorem 1]{Ntrees}}]\label{thm:anasolNB}
	Let the functions $\Fmu$ from \eqref{ivp} satisfy $\Fmu\in \mathcal C^{k+1}$ for $\mu=1,\dotsc,N$. If $\y$ solves \eqref{ivp}, then
	\[\y(t+h)=\NB_k(\tfrac{1}{\gamma},\y(t)) +\O(h^{k+1}),\]
	where $\gamma$ is the density defined in \eqref{eq:sigmagamma}.
\end{theorem}
The numerical solution given by one step of an ARK method can also be written as an NB-series $\NB(u,\y^n)$, with coefficients $u$ recursively determined by
\begin{equation}\label{eq:cond}
	\begin{aligned}
		u(\tau)&=\sum_{\nu=1}^N\sum_{i=1}^sb_i^{[\nu]} g_i^{[\nu]}(\tau),\\ 
		g_i^{[\nu]}(\rt[]^{[\mu]})&=\delta_{\nu\mu},&&  \nu,\mu=1,\dotsc,N,  \\
		g_i^{[\nu]}([\tau_1,\dotsc,\tau_l]^{[\mu]})&=\delta_{\nu\mu}\prod_{j=1}^ld_i(\tau_j),&& \nu,\mu=1,\dotsc,N\text{ and } 
		\\
		d_i(\tau)&=\sum_{\nu=1}^N\sum_{j=1}^sa_{ij}^{[\nu]} g_j^{[\nu]}(\tau),
	\end{aligned}
\end{equation}
see \cite{Ntrees}. From Theorem \ref{thm:anasolNB} and the fact that elementary differentials are linearly independent, to achieve an order at least $p$, we thus require
\begin{equation}
	u(\tau)=\frac{1}{\gamma(\tau)} \quad \text{ for all $\tau\in NT_p$.}\label{eq:c=1/gamma}
\end{equation}
\begin{rem}\label{rem:compute_u}
	Based on \cite[Lemma 312B]{B16}, the value of $u$ can be read off from a colored and labeled rooted tree $\tau$. There, a node labeled by $i$ and colored in $\mu$ is represented by $\ct[]_i^{[\mu]}$. It is convenient to also associate with each edge a color; we denote the edge connecting parent node $i$ to child node $j$ by $e_{ij}^{[v]}$, where $\nu$ is the color of node $j$.  We denote the set of labels by $L(\tau)$ and the set of colored edges by $E(\tau)$.
	
	For computing $u(\tau)$, let the root of $\tau$ be labeled by $i$ and colored in $\mu$. Then form the product \[b_{i}^{[\mu]}\prod_{e_{jk}^{[\nu]}\in E(\tau)} a_{jk}^{[\nu]}\] and sum over all elements of $L(\tau)$ ranging over the index set $\{1,\dotsc, s\}$. The result of the sum equals $u(\tau)$.
\end{rem}
\begin{example}
	We label the colored rooted tree $\tau=\left[\left[\rt[]^{[\xi]}\right]^{[\nu]}\right]^{[\mu]}$ and represent the result by
	\[\ct[\hphantom{.}_i^{[\mu]}[\hphantom{.}_j^{[\nu]}[\hphantom{.}_k^{[\xi]}]]]\]
	so that $u(\tau)=\sum_{i,j,k=1}^sb_i^{[\mu]} a_{ij}^{[\nu]} a_{jk}^{[\xi]}$ since $E(\tau)=\left\{e_{ij}^{[\nu]},e_{jk}^{[\xi]}\right\}$ and $L(\tau)=\{i,j,k\}$. 
	
	For the tree $\tau=\left[\rt[]^{[\nu]},\rt[]^{[\xi]}\right]^{[\mu]}$, which we label and represent by
	\[\ct[\hphantom{.}_i^{[\mu]}[\hphantom{.}_j^{[\nu]}][\hphantom{.}_k^{[\xi]}]]\]
	the value of $u(\tau)$ equals $\sum_{i,j,k=1}^sb_i^{[\mu]} a_{ij}^{[\nu]} a_{ik}^{[\xi]}$ as $E(\tau)=\left\{e_{ij}^{[\nu]},e_{ik}^{[\xi]}\right\}$ and $L(\tau)=\{i,j,k\}$. 
\end{example}

In the appendix, we prove modified versions of theorems from \cite{B16} to demonstrate that for an NSARK scheme we can take the formula for $u$ from \eqref{eq:cond} and replace the constant coefficients with the solution-dependent ones from \eqref{eq:nsark}, i.\,e.\ that the solution-dependent $u=u(\tau,\Y^n,h)$ in the case of an NSARK method is given by
\begin{equation}\label{eq:pertcond}
	\begin{aligned}
		u(\tau,\Y^n,h)&=\sum_{\nu=1}^N\sum_{i=1}^sb_i^{[\nu]}(\Y^n,h) g_i^{[\nu]}(\tau,\Y^n,h),\\ 
		g_i^{[\nu]}(\rt[]^{[\mu]},\Y^n,h)&=\delta_{\nu\mu},&& \nu,\mu=1,\dotsc,N, \\
		g_i^{[\nu]}([\tau_1,\dotsc,\tau_l]^{[\mu]},\Y^n,h)&=\delta_{\nu\mu}\prod_{j=1}^ld_i(\tau_j,\Y^n,h),&& \nu,\mu=1,\dotsc,N\text{ and } \\
		d_i(\tau,\Y^n,h)&=\sum_{\nu=1}^N\sum_{j=1}^sa_{ij}^{[\nu]}(\Y^n,h) g_j^{[\nu]}(\tau,\Y^n,h).
	\end{aligned}
\end{equation} 
As a result of this claim, we would be in the position to formulate an analogous condition to \eqref{eq:c=1/gamma} for an NSARK method to have an order of at least $p$.

\section{Main result} To prove our main result, we introduce in Theorem \ref{thm:main} a generalization of NB-series, in which the coefficients of the series are allowed to depend on $\y^n$ and $h$. We note that such a series is \emph{not} a Taylor expansion in $\dt$, but instead can be understood as an asymptotic expansion in expressions depending on powers of $\dt$ and the solution-dependent coefficients of the Butcher tableau. As a result of this approach, we do not require at this point any regularity of $a_{ij}^{[\nu]}(\Y^n,h)$ or $b^{[\nu]}_j(\Y^n,h)$, but only their boundedness as $h\to 0$. We want to mention here that we use the $\O$ notation always with respect to $h$ as $h\to 0$. Nevertheless, we still add this information where we think that the reader might otherwise be confused.

The results in this section are analogous to results in \cite{B16}, and we follow many of the ideas employed therein.
The proofs of the intermediate results can be found in the appendix, so that we directly present and prove the main theorem analogously to Theorem 313B in \cite{B16}. 
\begin{theorem}\label{thm:main}
	Let $d_i,g_i^{[\nu]}$ and $u$ be defined as in \eqref{eq:pertcond} for $i=1,\dotsc,s$ and $\nu=1,\dotsc, N$. Suppose that for small enough $h$ there exists a solution to the stage equations \eqref{eq:nsark} of the NSARK method,
	that, $\Fnu\in \mathcal C^{k+1}$ for $k\in \N$ is Lipschitz continuous,
	and that $a_{ij}^{[\nu]}(\Y^n,h)=\O(1)$ (with respect to $h$, as $h\to 0$) for all $\nu=1,\dotsc,N$.
	Then
	the stages, stage derivatives and output of the NSARK method can be expressed as
	\begin{subequations}
		\begin{align}
			\byi&=\y^n+\sum_{\tau\in NT_{k}}\frac{1}{\sigma(\tau)}d_i(\tau,\Y^n,h)h^{\lvert \tau\rvert} \dF(\tau)(\y^n)+\O(h^{k+1}),&& i=1,\dotsc, s,\label{eq:yithm}\\
			h\Fnu(\byi)&=\sum_{\tau\in NT_{k}}\frac{1}{\sigma(\tau)}g^{[\nu]}_i(\tau,\Y^n,h)h^{\lvert \tau\rvert} \dF(\tau)(\y^n)+\O(h^{k+1}),&& i=1,\dotsc, s,\quad \nu=1,\dotsc,N,\label{eq:hFlthm}\\
			\y^{n+1}&=\y^n+\sum_{\tau\in NT_{k}}\frac{1}{\sigma(\tau)}u(\tau,\Y^n,h)h^{\lvert \tau\rvert} \dF(\tau)(\y^n)+\O(h^{k+1}).\label{eq:expyn+1}
		\end{align}
	\end{subequations}
\end{theorem}
\begin{proof}
	We follow the idea from \cite[Theorem 313B]{B16}. For approximating the stage $\byi$, define the sequence
	\begin{equation}\label{eq:byim}
		\begin{aligned}
			\byi_{[0]}&=\y^n,\\
			\byi_{[m]}&=\y^n+h\sum_{j=1}^s\sum_{\nu=1}^Na_{ij}^{[\nu]}(\Y^n,h)\Fnu(\byj_{[m-1]}),
		\end{aligned}
	\end{equation}
	where we want to point out that $a_{ij}^{[\nu]}(\Y^n,h)$ here only depends on the solution $\y^n$, the step size $h$ and, potentially, the assumed solution to the stage equations, but \emph{not} on the iterates $\byi_{[m]}$.
	
	Next, we demonstrate that for $m\leq k$, this expression for $\byi_{[m]}$ agrees with the expression for $\byi$ from \eqref{eq:yithm} within an error of $\O(h^{m+1})$. For $m=0$, this is obvious. By induction we suppose that \[\byi_{[m-1]}=\y^n+\sum_{\tau\in NT_{m-1}}\frac{1}{\sigma(\tau)}d_i(\tau,\Y^n,h)h^{\lvert \tau\rvert} \dF(\tau)(\y^n)+\O(h^{m}).\] By Lemma \ref{lem:hfnu}, we see that 
	\[h\Fnu(\byi_{[m-1]})=\sum_{\tau\in NT_{m}}\frac{1}{\sigma(\tau)}g^{[\nu]}_i(\tau,\Y^n,h)h^{\lvert \tau\rvert} \dF(\tau)(\y^n)+\O(h^{m+1}).\]
	Substituting this into \eqref{eq:byim}, we see from \eqref{eq:pertcond} that
	\begin{equation}\label{eqthm:yim}
		\begin{aligned}
			\byi_{[m]}&=\y^n+\sum_{\tau\in NT_{m}}\frac{1}{\sigma(\tau)}\sum_{j=1}^s\sum_{\nu=1}^Na_{ij}^{[\nu]}(\Y^n,h)g^{[\nu]}_j(\tau,\Y^n,h)h^{\lvert \tau\rvert} \dF(\tau)(\y^n)+\O(h^{m+1}) \\
			&=\y^n+\sum_{\tau\in NT_{m}}\frac{1}{\sigma(\tau)}d_i(\tau,\Y^n,h)h^{\lvert \tau\rvert} \dF(\tau)(\y^n)+\O(h^{m+1}).
		\end{aligned}
	\end{equation}
	We have shown now that \eqref{eqthm:yim} is true for all $m\leq k$. Indeed, by the same reasoning we have even proven that 
	\[\byi_{[m]}=\y^n+\sum_{\tau\in NT_{k}}\frac{1}{\sigma(\tau)}d_i(\tau,\Y^n,h)h^{\lvert \tau\rvert} \dF(\tau)(\y^n)+\O(h^{k+1}) \quad\text{ for all $m\geq k$. }\]
	
	Moreover, for $h$ small enough we know that $a_{ij}^{[\nu]}(\Y^n,h)$ is bounded since we assumed $a^{[\nu]}_{ij}(\Y^n,h)=\O(1)$ as $h\to 0$.
	Together with the Lipschitz continuity of $\Fnu$, we thus conclude that for small enough $h$ the iteration \eqref{eq:byim} is a contraction with $\byi=\lim_{m\to \infty}\byi_{[m]}$ being the unique limit. Thus, for $h$ small enough and $\epsilon=h^{k+1}>0$, there exist $N_\epsilon\in \N$ such that $\Vert \byi_{[m]}-\byi\Vert<h^{k+1}$ for all $m\geq N_\epsilon$. Without loss of generality we can choose $N_\epsilon\geq k$, so that we find $m\geq k$. This implies that 
	\[\byi=\byi_{[m]}+\O(h^{k+1})=\y^n+\sum_{\tau\in NT_{k}}\frac{1}{\sigma(\tau)}d_i(\tau,\Y^n,h)h^{\lvert \tau\rvert} \dF(\tau)(\y^n)+\O(h^{k+1}),\] from which
	equation \eqref{eq:yithm} follows. Furthermore, \eqref{eq:hFlthm} then follows from Lemma \ref{lem:hfnu}. Finally, computing $\y^{n+1}$ according to \eqref{eq:nsark}, also taking into account equation \eqref{eq:pertcond}, we obtain
	\begin{equation*}
		\begin{aligned}
			\y^{n+1} &= \y^n +  \sum_{j=1}^s \sum_{\substack{\nu=1}}^N b^{[\nu]}_j(\Y^n,h) h\Fnu(\byj)\\
			&=\y^n+\sum_{\tau\in NT_{k}}\frac{1}{\sigma(\tau)}\sum_{j=1}^s \sum_{\substack{\nu=1}}^Nb^{[\nu]}_j(\Y^n,h)g^{[\nu]}_j(\tau,\Y^n,h)h^{\lvert \tau\rvert} \dF(\tau)(\y^n)+\O(h^{k+1})\\
			&=\y^n+\sum_{\tau\in NT_{k}}\frac{1}{\sigma(\tau)}u(\tau,\Y^n,h)h^{\lvert \tau\rvert} \dF(\tau)(\y^n)+\O(h^{k+1}),
		\end{aligned}
	\end{equation*}
	finishing the proof.
\end{proof}
Note that under the assumptions of this theorem, any solution of the stage equations has the same Taylor expansion up to the order $k$. Moreover, we obtain the following order conditions as a result of this theorem, where we the expression $\A^{[\nu]}(\Y^n,h)=(a^{[\nu]}_{ij}(\Y^n,h))_{i,j=1,\dotsc,s}=\O(1)$ should be understood component-wise and in the limit $h\to 0$.
\begin{corollary}\label{cor:orderNSARK}
	Let $u$ be defined as in \eqref{eq:pertcond} and the stage equations of the NSARK method possess a solution for small enough $h$. Furthermore, let $\A^{[\nu]}(\Y^n,h)=\O(1)$ and $\Fnu\in \mathcal C^{p+1}$ for $\nu=1,\dotsc,N$ be Lipschitz continuous. Then, an NSARK scheme \eqref{eq:nsark} is of order at least $p$ if and only if 
	\begin{equation}
		u(\tau,\Y^n,h)=\frac{1}{\gamma(\tau)}+\O(h^{p+1-\lvert \tau \rvert}), \quad \forall \tau\in NT_p.\label{eq:u=1:gamma}
	\end{equation}
\end{corollary}

\begin{corollary}\label{cor:suff_cond} Let $\A^{[\nu]}=(a^{[\nu]}_{ij})_{i,j=1,\dotsc,s},$ $\b^{[\nu]}=(b^{[\nu]}_1,\dotsc,b^{[\nu]}_N)$ denote the coefficients of an ARK method of order $p$, and let
	\begin{equation}\label{eq:suff_cond}
		a_{ij}^{[\nu]}(\Y^n,\dt)=a^{[\nu]}_{ij}+\O(\dt^{p-1}) \qta b_j^{[\nu]}(\Y^n,\dt)=b^{[\nu]}_j+\O(\dt^p),
	\end{equation}
	for $i,j=1,\dotsc,s$ and $\nu=1,\dotsc,N$, and suppose that all the assumptions stated in Theorem~\ref{thm:main} hold. Then the NSARK method \eqref{eq:nsark}  applied to autonomous problems is of order $p$. 
\end{corollary}
\begin{proof}
	Inserting \eqref{eq:suff_cond} into \eqref{eq:nsark} yields
	\begin{equation*}
		\begin{aligned}
			\byi & = \y^n + \dt \sum_{j=1}^s  \sum_{\substack{\nu=1}}^N a^{[\nu]}_{ij}  \fnu(\byj) +\O(\dt^p), \\
			\y^{n+1} & = \y^n + \dt \sum_{j=1}^s \sum_{\substack{\nu=1}}^N b^{[\nu]}_j \fnu(\byj) +\O(\dt^{p+1}).
		\end{aligned}
	\end{equation*}
	According to Theorem~\ref{thm:main} and  Lemma~\ref{lem:hfnu} we see 
	\[\dt\fnu(\byi)=\sum_{\tau\in NT_{p}}\frac{\dt^{\lvert \tau\rvert}}{\sigma(\tau)}g^{[\nu]}_i(\tau) \dF(\tau)(\y^n)+\O(\dt^{p+1}). \]
	Consequently, \eqref{eq:cond} implies that \[\y^{n+1}=\y^n+\sum_{\tau\in NT_{p}}\frac{\dt^{\lvert \tau\rvert}}{\sigma(\tau)}u(\tau) \dF(\tau)(\y^n)+\O(\dt^{p+1}).\]
	Finally, since the corresponding underlying ARK scheme is of order $p$ the claim follows.
\end{proof}

In order to elucidate the condition \eqref{eq:u=1:gamma} from Corollary~\ref{cor:orderNSARK}, we collect the value of $u$ for  all $\tau\in NT_4$ in Table~\ref{tab:NSARKcond}.
\begin{table}[!h]
	\centering
	{\def\arraystretch{2}\tabcolsep=10pt	\begin{tabular}{|>{\centering\arraybackslash}m{0.13\linewidth}|>{\centering\arraybackslash}m{.05\linewidth}|>{\centering\arraybackslash}m{.6\linewidth}|}
			\hline
			$\tau$          &  $\gamma(\tau)$ & $u(\tau,\Y^n,\dt)$ \\ \hline
			$\rt[]^{[\mu]}$ &   1                       &     $\sum_{i=1}^s b_i^{[\mu]}(\Y^n,\dt)$               \\ \hline
			$\ct[\hphantom{.}^{[\mu]}[\hphantom{.}^{[\nu]}]]$	&   2                        &   $\sum_{i,j=1}^s b_i^{[\mu]}(\Y^n,\dt) a_{ij}^{[\nu]}(\Y^n,\dt)$                 \\ \hline
			$	\ct[\hphantom{.}^{[\mu]}[\hphantom{.}^{[\nu]}[\hphantom{.}^{[\xi]}]]]$              &      6        &     $\sum_{i,j,k=1}^s b_i^{[\mu]}(\Y^n,\dt) a_{ij}^{[\nu]}(\Y^n,\dt)a_{jk}^{[\xi]}(\Y^n,\dt)$               \\ \hline
			$\ct[\hphantom{.}^{[\mu]}[\hphantom{.}^{[\nu]}][\hphantom{.}^{[\xi]}]]$	          &         3     &   $\sum_{i,j,k=1}^s b_i^{[\mu]}(\Y^n,\dt) a_{ij}^{[\nu]}(\Y^n,\dt)a_{ik}^{[\xi]}(\Y^n,\dt)$                 \\ \hline
			$	\ct[\hphantom{.}^{[\mu]}[\hphantom{.}^{[\nu]}[\hphantom{.}^{[\xi]}[\hphantom{.}^{[\eta]}]]]]$              &  24            &    $	\smashoperator{\sum_{i,j,k,l=1}}^s b_i^{[\mu]}(\Y^n,\dt) a_{ij}^{[\nu]}(\Y^n,\dt)a_{jk}^{[\xi]}(\Y^n,\dt)a_{kl}^{[\eta]}(\Y^n,\dt)$                \\ \hline
			$\ct[\hphantom{.}^{[\mu]}[\hphantom{.}^{[\eta]}][\hphantom{.}^{[\nu]}][\hphantom{.}^{[\xi]}]]$	             &  4            &     $		\smashoperator{\sum_{i,j,k,l=1}}^s b_i^{[\mu]}(\Y^n,\dt)a_{il}^{[\eta]}(\Y^n,\dt) a_{ij}^{[\nu]}(\Y^n,\dt)a_{ik}^{[\xi]}(\Y^n,\dt)$               \\ \hline
			$	\ct[\hphantom{.}^{[\mu]}[\hphantom{.}^{[\xi]}[\hphantom{.}^{[\eta]}]][\hphantom{.}^{[\nu]}]]$              &   8           &    $	\smashoperator{\sum_{i,j,k,l=1}}^s b_i^{[\mu]}(\Y^n,\dt)a_{il}^{[\nu]}(\Y^n,\dt) a_{ij}^{[\xi]}(\Y^n,\dt)a_{jk}^{[\eta]}(\Y^n,\dt)$                \\ \hline
			$	\ct[\hphantom{.}^{[\mu]}[\hphantom{.}^{[\nu]}[\hphantom{.}^{[\xi]}][\hphantom{.}^{[\eta]}]]]$              &   12           &    $	\smashoperator{\sum_{i,j,k,l=1}}^s b_i^{[\mu]}(\Y^n,\dt) a_{ij}^{[\nu]}(\Y^n,\dt)a_{jk}^{[\xi]}(\Y^n,\dt)a_{jl}^{[\eta]}(\Y^n,\dt)$                \\ \hline
	\end{tabular}}\caption{Density $\gamma$ from \eqref{eq:sigmagamma} and value of $u$ from \eqref{eq:pertcond} for $\tau\in NT_4$.}\label{tab:NSARKcond}
\end{table}
\begin{rem}
	Using Corollary \ref{cor:orderNSARK} and Table~\ref{tab:NSARKcond}, the condition for $p=1$ reads
	\begin{equation}\label{eq:condp=1}
		\begin{aligned}
			\sum_{i=1}^s b_i^{[\mu]}(\Y^n,\dt)=1 +\O(\dt),&& \mu=1,\dotsc,N.
		\end{aligned}
	\end{equation}
	For $p=2$ we find the conditions
	\begin{equation}\label{eq:condp=2}
		\begin{aligned}
			\sum_{i=1}^s b_i^{[\mu]}(\Y^n,\dt)&=1 +\O(\dt^2),& \mu&=1,\dotsc,N,\\
			\sum_{i,j=1}^s b_i^{[\mu]}(\Y^n,\dt) a_{ij}^{[\nu]}(\Y^n,\dt)&=\frac12 +\O(\dt),& \mu,\nu&=1,\dotsc,N,\\
		\end{aligned}
	\end{equation}
	and for $p=3$ we obtain
	\begin{equation}\label{eq:condp=3}
		\begin{aligned}
			\sum_{i=1}^s b_i^{[\mu]}(\Y^n,\dt)&=1 +\O(\dt^3), &\mu&=1,\dotsc,N,\\
			\sum_{i,j=1}^s b_i^{[\mu]}(\Y^n,\dt) a_{ij}^{[\nu]}(\Y^n,\dt)&=\frac12 +\O(\dt^2), &\mu,\nu&=1,\dotsc,N,\\
			\sum_{i,j,k=1}^s b_i^{[\mu]}(\Y^n,\dt) a_{ij}^{[\nu]}(\Y^n,\dt)a_{ik}^{[\xi]}(\Y^n,\dt)&=\frac13 +\O(\dt), &\mu,\nu,\xi&=1,\dotsc,N,\\
			\sum_{i,j,k=1}^s b_i^{[\mu]}(\Y^n,\dt) a_{ij}^{[\nu]}(\Y^n,\dt)a_{jk}^{[\xi]}(\Y^n,\dt)&=\frac16 +\O(\dt), &\mu,\nu,\xi&=1,\dotsc,N.
		\end{aligned}
	\end{equation}
	As we derive also 4th order conditions for GeCo and MPRK schemes in the next sections, we also present the general conditions for $p=4$ reading
	\begin{equation}\label{eq:condp=4}
		\begin{aligned}
			\sum_{i=1}^s b_i^{[\mu]}(\Y^n,\dt)&=1 +\O(\dt^4),\\
			\sum_{i,j=1}^s b_i^{[\mu]}(\Y^n,\dt) a_{ij}^{[\nu]}(\Y^n,\dt)&=\frac12 +\O(\dt^3),\\
			\sum_{i,j,k=1}^s b_i^{[\mu]}(\Y^n,\dt) a_{ij}^{[\nu]}(\Y^n,\dt)a_{ik}^{[\xi]}(\Y^n,\dt)&=\frac13 +\O(\dt^2), \\
			\sum_{i,j,k=1}^s b_i^{[\mu]}(\Y^n,\dt) a_{ij}^{[\nu]}(\Y^n,\dt)a_{jk}^{[\xi]}(\Y^n,\dt)&=\frac16 +\O(\dt^2), \\
			\sum_{i,j,k,l=1}^s b_i^{[\mu]}(\Y^n,\dt)a_{il}^{[\nu]}(\Y^n,\dt) a_{ij}^{[\xi]}(\Y^n,\dt)a_{jk}^{[\eta]}(\Y^n,\dt)&=\frac18 +\O(\dt), \\
			\sum_{i,j,k,l=1}^s b_i^{[\mu]}(\Y^n,\dt)a_{il}^{[\eta]}(\Y^n,\dt) a_{ij}^{[\nu]}(\Y^n,\dt)a_{ik}^{[\xi]}(\Y^n,\dt)&=\frac14 +\O(\dt), \\
			\sum_{i,j,k,l=1}^s b_i^{[\mu]}(\Y^n,\dt) a_{ij}^{[\nu]}(\Y^n,\dt)a_{jk}^{[\xi]}(\Y^n,\dt)a_{kl}^{[\eta]}(\Y^n,\dt)&=\frac{1}{4!} +\O(\dt), \\
			\sum_{i,j,k,l=1}^s b_i^{[\mu]}(\Y^n,\dt) a_{ij}^{[\nu]}(\Y^n,\dt)a_{jk}^{[\xi]}(\Y^n,\dt)a_{jl}^{[\eta]}(\Y^n,\dt)&=\frac{1}{12} +\O(\dt)
		\end{aligned}
	\end{equation}
	for $\mu,\nu,\xi,\eta=1,\dotsc,N$.
\end{rem}

\section{Application to GeCo Schemes} \label{sec:GeCo}
In this section we derive the known order conditions for Geometric Conservative (GeCo) schemes and present for the first time order conditions for 3rd and 4th order. 
GeCo schemes applied to \eqref{ivp} with $\F=\sum_{\nu=1}^N\Fnu$ are based on explicit RK schemes and take the form \cite{martiradonna2020geco} 
\begin{equation}\label{eq:GeCoscheme}
	\begin{aligned}
		\byi&=\y^n+\phi_i(\Y^n,h) h\sum_{j=1}^{i-1}a_{ij}\F(\byj),\quad i=1,\dotsc,s,\\
		\y^{n+1}&=\y^n+\phi_{n+1}(\Y^n,h) h\sum_{j=1}^sb_j\F(\byj).
	\end{aligned}
\end{equation}
Note that $\phi$ here corresponds to the function $\Phi$ of \cite{martiradonna2020geco} divided by $h$, and that the value of $\phi_1$ has no effect since $a_{1j}=0$.
The idea is to choose the functions  $\phi_i$ and $\phi_{n+1}$
in a way that guarantees the positivity of the stages
and the updated solution.  At the same time, these functions
must be chosen in a way that does not compromise the order
of accuracy.
Up to now, only conditions for first and second order GeCo schemes are available.  Furthermore, due to the complex structure of the functions $\phi_i$ and $\phi_{n+1}$ from \cite{martiradonna2020geco}, a stability analysis was not available until recently \cite{gecostab}.  

To study the order of the GeCo scheme \eqref{eq:GeCoscheme},
we absorb the factors $\phi_i,\, \phi_{n+1}$ into the RK coefficients, leading to
a non-standard RK (NSRK) scheme, which we can write formally in the notation of the
previous two sections via the coefficients:
\[a^{[1]}_{ij}(\Y^n,h)=a_{ij}\phi_i(\Y^n,h),\quad b^{[1]}_j(\Y^n,h)=b_j\phi_{n+1}(\Y^n,h),\quad i,j=1,\dotsc, s.\]
\begin{rem}
	We want to emphasize two things at this point. First, in the context of GeCo methods, Corollary~\ref{cor:suff_cond} generalizes \cite[Theorem 3]{DH20} as we do not require $\phi_i=\phi_{n+1}$. Second, not all schemes of the form \eqref{eq:GeCoscheme} are GeCo methods as one would additionally require that the numerical solution is positivity preserving. Nevertheless, this aspect has no influence on the order conditions we aim to derive. 
\end{rem}
We can then obtain the order conditions from Corollary \ref{cor:orderNSARK}.
These can easily be simplified somewhat, using the fact that the original
coefficients $a_{ij},\, b_j$ satisfy traditional RK order conditions.  The first condition is
$$
\sum_{i=1}^s b_i \phi_{n+1}(\Y^n,h) = 1 + \O(h^{p})
$$
which implies simply $\phi_{n+1}(\Y^n,h) = 1 + \O(h^p)$.  This turns out
to allow us to neglect the factor $\phi_{n+1}$ in all the remaining
order conditions.  For instance, using $c_i=\sum_{j=1}^sa_{ij}$, the next condition is
$$
\sum_{i=1}^s b_i c_i \phi_{n+1}(\Y^n,h) \phi_i(\Y^n,h) = \frac 12 + \O(h^{p-1}),
$$
which is equivalent to 
\[\sum_{i=1}^s b_i c_i \phi_i(\Y^n,h) =\frac12 +\O(h^{p-1}).\]
With more work, we can use these conditions to obtain direct conditions
on the functions $\phi$ for specific cases of $s$ and $p$, as demonstrated
in the following theorem.
\begin{theorem}\label{thm:geco_old}
	Let $\A,\b$ be the coefficients of an explicit RK scheme of order $p$ with $s$ stages. Assume $\phi_i(\Y^n,h)=\O(1)$ as $h\to 0$ for $i=2,\dotsc,s$ and that $\F\in \mathcal C^{p+1}$ is Lipschitz continuous. Then
	\begin{enumerate}
		\item if $p=1$, then the method \eqref{eq:GeCoscheme}  is of order $1$ if and only if $\phi_{n+1}(\Y^n,h)=1+\O(h)$.
		\item if $p=s=2$, then the method \eqref{eq:GeCoscheme} is of order $2$ if and only if $\phi_2(\Y^n,h)=1+\O(h)$ and $ \phi_{n+1}(\Y^n,h)=1+\O(h^2)$.
			\end{enumerate}
		\end{theorem}
		\begin{proof} First of all, the assumptions of Theorem \ref{thm:main} and Corollary \ref{cor:orderNSARK} are met. Thus, we can use the order conditions \eqref{eq:condp=1} to \eqref{eq:condp=4} as a basis of this proof.
			\begin{enumerate}
				\item Substituting $\sum_{i=1}^sb_i=1$ into \eqref{eq:condp=1} yields $u(\rt[]^{[1]},\Y^n,h)=\phi_{n+1}(\Y^n,h)=1+\O(h)$.
				\item Using $\sum_{i=1}^sb_i=1$ now in \eqref{eq:condp=2} together with $\sum_{j=1}^s a_{ij}=c_i$, the order conditions reduce to 
				\[ u(\rt[]^{[1]},\Y^n,h)=\phi_{n+1}(\Y^n,h)=1+\O(h^2),\]
				and
				\[ u([\rt[]^{[1]}]^{[1]},\Y^n,h)=\sum_{i=1}^s b_i\phi_{n+1}(\Y^n,h)c_i\phi_i(\Y^n,h)=\frac12 +\O(h).\]
				The latter condition can be further simplified to 
				\[b_2c_2\phi_2(\Y^n,h)=\frac12 +\O(h),\] since $s=2$, $c_1=0$, and $\phi_{n+1}(\Y^n,h)=1+\O(h^2)$. As $b_2c_2=\frac12$, this means that 
				\[ \phi_2(\Y^n,h)=1+\O(h).\]
			\end{enumerate}
		\end{proof}
		After reproducing known order conditions, we now present for the first time the necessary and sufficient conditions for GeCo methods of order three and four.
		\begin{theorem}\label{thm:geco_new}
			Let the assumptions of Theorem~\ref{thm:geco_old} hold with given values of $p$ and $s$. Then the method \eqref{eq:GeCoscheme} with
			\begin{enumerate}
				\item $p=s=3$ is of order 3 if and only if    \begin{align*}
					\phi_{n+1}(\Y^n,h)&=1 +\O(h^3),\\
					\sum_{i=2}^3 b_ic_i\phi_i(\Y^n,h)&=\frac 1 2 +\O(h^2),\\
					\phi_i(\Y^n,h)&=1 +\O(h),\quad i=2,3.
				\end{align*}
				\item with $p=s=4$ is of order $4$ if and only if    \begin{align*}
					\phi_{n+1}(\Y^n,h)&=1 +\O(h^4),\\
					\sum_{i=2}^4 b_ic_i\phi_i(\Y^n,h)&=\frac 1 2 +\O(h^3),\\
					\phi_i(\Y^n,h)&=1 +\O(h^2),\quad i=2,3,4.
				\end{align*}
			\end{enumerate}
		\end{theorem}
		\begin{proof}
			\begin{enumerate}
				\item Similar as in the proof of Theorem~\ref{thm:geco_old}, we obtain from \eqref{eq:condp=3} the simplified conditions
				\begin{align*}
					\phi_{n+1}(\Y^n,h)&=1 +\O(h^3),\\
					\sum_{i=2}^3 b_ic_i\phi_i(\Y^n,h)&=\frac12 +\O(h^2),\\
					\sum_{i=2}^3 b_ic_i^2(\phi_i(\Y^n,h))^2&=\frac13 +\O(h), \\
					\sum_{i,j=2}^3 b_i a_{ij}c_j\phi_i(\Y^n,h)\phi_j(\Y^n,h)&=\frac16 +\O(h), 
				\end{align*}
				which by Lemma \ref{lem:equivalent} with $N=1$, $\gamma^{(i)}_1=\phi_i(\Y^n,h)$ and $\gamma^{n+1}_1=\phi_{n+1}(\Y^n,h)$ are equivalent to
				the conditions stated in the Theorem.
					\item As in the previous part, the conditions \eqref{eq:condp=4} are simplified resulting in \eqref{eq:MPRKcondp=4} with $N=1$, $\gamma^{(i)}_1$ and $\gamma^{n+1}_1$ as before. These conditions are then reduced by Lemma \ref{lem:equivalent4} resulting in the conditions given in this theorem.
				\end{enumerate}
			\end{proof}
			With this result, we have shown that the conditions from \cite[Theorem 1]{martiradonna2020geco} are also necessary. Moreover, we provided the very first necessary and sufficient order conditions for the construction of 3rd and 4th order GeCo schemes.
			
			\section{Application to MPRK Schemes} \label{sec:MPRK}
			Modified Patankar--Runge--Kutta methods (MPRK) are applied to so-called production-destruction systems (PDS)
			\begin{align}\label{eq:PDS}
				y_m'(t) & = \sum_{\nu=1}^N \left(p_{m\nu}(\y(t)) - d_{m\nu}(\y(t)) \right),\quad m=1,\dotsc, N,
			\end{align}
			which are absolutely conservative, i.\,e.\ $p_{m\nu}=d_{\nu m}$ and $p_{mm}=d_{mm}=0$ for $m,\nu=1,\dotsc, N$. This property implies that $\sum_{m=1}^Ny_m'(t)=0$ resulting in $\sum_{m=1}^Ny(t)=\sum_{m=1}^Ny(0)$ for all times $t\geq 0$.  We note that $\y(t)\in \R^d$ with $d=N$ in this case, see \eqref{ivp}.
			
			MPRK schemes belong to the class of Patankar-type methods, which preserve the positivity of the PDS for any step size $h>0$ \cite{MPRK2,MPRK3}. In addition to being unconditionally positive, MPRK schemes are also unconditionally conservative in the sense that $\sum_{m=1}^Ny_m^n=\sum_{m=1}^Ny_m^0$ for all $n\in \N$. Besides these properties, MPRK schemes demonstrated their robustness many times by solving also stiff problems \cite{BBHBS,MPRK2,MPRK3}. Recently, a Lyapunov stability analysis has been published underlining their robustness observed in numerical experiments \cite{IKM2,IKMsys,IOE22StabMP}.
			
			MPRK schemes applied to \eqref{eq:PDS} are based on explicit RK methods with a non-negative Butcher tableau and are of the form
			\begin{subequations}\label{eq:MPRK}
				\begin{align} \label{MPRKstage}
					\yi_m & = y^n_m + h \sum_{j=1}^{i-1} a_{ij} \sum_{\nu=1}^N \left( p_{m\nu}(\byj) \frac{\yi_\nu}{\rho_{i\nu}} - d_{m\nu}(\byj)\frac{\yi_m}{\rho_{im}} \right),  &m=1,\dotsc,N,& & i=1,\dotsc,s, \\
					y^{n+1}_m & = y^n_m + h \sum_{j=1}^s b_j \sum_{\nu=1}^N \left( p_{m\nu}(\byj)\frac{y^{n+1}_\nu}{\sigma_\nu} - d_{m\nu}(\byj) \frac{y^{n+1}_m}{\sigma_m} \right), &m=1,\dotsc,N. &&
				\end{align}
			\end{subequations}
			Here, $\rho_{i\nu}$ and $\sigma_\nu$ are called Patankar-weight denominators (PWDs). They are independent of the corresponding numerator in \eqref{eq:MPRK} for all $i=2,\dotsc,s$, $\nu=1,\dotsc, N$ and  required to be positive for all $h\geq 0$. Examples of  $\rho_{i\nu}$ and $\sigma_\nu$ can be found in \cite{BBHBS,MPRK2,MPRK3}. In the particular case of the Heun method as a starting point, one could choose $\rho_{2\nu}=y_\nu^n$ and $\sigma_\nu=y_\nu^{(2)}$. In general, the PWDs presented so far were chosen to be continuous functions of $\y^n$ and the stages, that is $\rho_{i\nu}=\rho_{i\nu}(\y^n,\y^{(1)},\dotsc,\y^{(i-1)})$ and $\sigma_\nu=\sigma_\nu(\y^n,\y^{(1)},\dotsc,\y^{(s)})$. 
			
			\begin{rem}\label{rem:negweights}
				As mentioned above, the MPRK scheme \eqref{eq:MPRK} is formulated for non-negative Runge--Kutta parameters.
				But MPRK schemes with negative Runge--Kutta parameters can be devised as well.
				In this case, the weighting of the production and destruction terms which get multiplied by the negative weight must be interchanged. To be precise, the index $\nu$ of the PWDs $\rho_{i\nu}$ and $\sigma_\nu$ in the formula \eqref{eq:MPRK} is replaced by the value of the \emph{index function} 
				\begin{equation}\label{eq:indexfun}
					\gamma(\nu,m,x)=\begin{cases}
						\nu, &x\geq 0\\
						m, & x< 0
					\end{cases}
				\end{equation} at $x=a_{ij}$ and $x=b_j$, respectively. Similarly, the index $m$ is replaced by $\gamma(m,\nu,a_{ij})$ and $\gamma(m,\nu,b_j)$ for $\rho_{im}$ and $\sigma_m$, respectively.
				
				This procedure will ensure the unconditional positivity of the scheme, but one may argue that this has an impact on the necessary requirements to obtain a certain order of accuracy. Fortunately, we will discover that this is not the case in Section~\ref{sec:MPRK}. To avoid multiple case distinctions we require that all Runge--Kutta coefficients be positive in the remainder of this work.
			\end{rem}
			
			Until now, sufficient and necessary order conditions for MPRK schemes only up to order three were constructed. In order to obtain order conditions for even higher order MPRK schemes, we first introduce \[F^{[\nu]}_m(\y(t))=\begin{cases}p_{m\nu}(\y(t)), &\nu\neq m,\\-\sum_{j=1}^Nd_{mj}(\y(t)), &\nu=m \end{cases}\] and see, due to $p_{mm}=0$, that the PDS \eqref{eq:PDS} takes the form
			\[ y_m'(t)=\sum_{\substack{\nu=1\\ \nu\neq m}}^Np_{m\nu}(\y(t))-\sum_{j=1}^Nd_{mj}(\y(t))=\sum_{\substack{\nu=1\\ \nu\neq m}}^NF_m^{[\nu]}(\y(t))+F_m^{[m]}(\y(t))=\sum_{\nu=1}^NF_m^{[\nu]}(\y(t)).\]
			As a result, we are able to write the MPRK scheme \eqref{eq:MPRK} as an NSARK method according to
			\begin{align*} 
				\yi_m & = y^n_m + h \sum_{j=1}^{i-1} a_{ij} \left(\sum_{\substack{\nu=1\\ \nu\neq m}}^N  F_m^{[\nu]}(\byj) \frac{\yi_\nu}{\rho_{i\nu}} + F_m^{[m]}(\byj)\frac{\yi_m}{\rho_{im}} \right) \\
				&=y^n_m + h \sum_{j=1}^{i-1} \sum_{\substack{\nu=1}}^N a_{ij}\frac{\yi_\nu}{\rho_{i\nu}}  F_m^{[\nu]}(\byj) =y^n_m + h \sum_{j=1}^{i-1} \sum_{\substack{\nu=1}}^N a^{[\nu]}_{ij}(\Y^n,h)  F_m^{[\nu]}(\byj), \quad i=1,\dotsc,s, \\
				y^{n+1}_m & = y^n_m + h \sum_{j=1}^s b_j \left( \sum_{\substack{\nu=1\\ \nu\neq m}}^N  F_m^{[\nu]}(\byj)\frac{y^{n+1}_\nu}{\sigma_\nu}+ F_m^{[m]}(\byj) \frac{y^{n+1}_m}{\sigma_m} \right)= y^n_m + h \sum_{j=1}^s\sum_{\substack{\nu=1}}^N b^{[\nu]}_j(\Y^n,h)   F_m^{[\nu]}(\byj)
			\end{align*}
			with 
			\begin{equation}\label{eq:pertcoeff}
				a_{ij}^{[\nu]}(\Y^n,h)=a_{ij}\frac{\yi_\nu}{\rho_{i\nu}} \quad \text{ and } \quad b_j^{[\nu]}(\Y^n,h)=b_j\frac{y^{n+1}_\nu}{\sigma_\nu}
			\end{equation}
			being solution-dependent coefficients.  Note that, since $a_{1j}=0$, the value of
			$\rho_{1\nu}$ has no effect.
			
			In order to apply Theorem~\ref{thm:main} and Corollary~\ref{cor:orderNSARK}, we show in the next lemma that the stages are uniquely determined for any $h\geq 0$ and that $\A^{[\nu]}(\Y^n,h)=\O(1)$. The key observation to prove this is that $\pi_\nu^{(i)},\sigma_\nu$ are positive even for $h = 0$ by definition.  Also, as we will apply the lemma in the context of the local error analysis, we may start with some $\y^n$ representing the exact solution at a given time level $t_n$.
			\begin{lemma}\label{lem:aij=O(1)} An MPRK scheme \eqref{eq:MPRK} with a given positive vector $\y^n$ has uniquely determined stages and satisfies $\byi,\y^{n+1}=\O(1)$. If $\Fnu\in\mathcal C$ and $\rho_{i\nu}(\y^n,\y^{(1)},\dotsc,\y^{(i-1)}),\,\sigma_\nu(\y^n,\y^{(1)},\dotsc,\y^{(s)})>0$ are continuous functions of $\y^n$ and the stages, then $\rho_{i\nu},\,\sigma_\nu=\O(1)$ for $i,\nu=1,\dotsc,N$.
				Moreover, the modified Butcher coefficients from \eqref{eq:pertcoeff} satisfy $\A^{[\nu]}(\Y^n,h)=\O(1)$ and $\b^{[\nu]}(\Y^n,h)=\O(1)$.
			\end{lemma}
			\begin{proof}
				According to \cite[Lemma 3.1]{MPRK2}, there exist unique matrices $\mathbf M^{(i)}$, $i=1,\dotsc,s$ and $\mathbf{M}$, with inverses in $\O(1)$ as $h\to 0$, such that the stages satisfy $\byi=\left(\mathbf M^{(i)}\right)^{-1}\y^n=\O(1)$ and $\y^{n+1}=\mathbf{M}^{-1}\y^n=\O(1)$. Now, since $\Fnu,\rho_{i\nu},\sigma_\nu\in\mathcal C$, we conclude by induction over $i$ that the stage vectors are continuous functions of $h$ themselves by pointing out that every entry in $\mathbf M^{(i)},\mathbf M$ is a continuous function of $h$. Hence, even $\rho_{i\nu}$ and $\sigma_\nu$ are continuous functions of $h$, so that we conclude $\rho_{i\nu}=\O(1)$ and $\sigma_{\nu}=\O(1)$ as $h\to 0$. Since even $\rho_{i\nu},\sigma_\nu>0$ for $h=0$, we deduce from the continuity in $h$ that there is a positive lower bound also for $h>0$ small enough. This gives us $\frac{\yi_\nu}{\rho_{i\nu}}=\O(1)$ and  $\frac{\y^{n+1}}{\sigma_{\nu}}=\O(1)$, from which the claim follows. 
			\end{proof}
			Using this lemma we can apply Corollary \ref{cor:orderNSARK}, and hence, present the very first comprehensive theory for deriving order conditions of MPRK schemes up to an arbitrary order. 
			Although Corollary~\ref{cor:orderNSARK} provides necessary and sufficient conditions for arbitrary high order NSARK schemes, to which MPRK methods belong, those conditions are in general implicit, since they depend on the stages.  In the next two subsections, for specific classes of MPRK methods we reformulate these conditions to be explicit.
			
			\begin{rem}\label{rem:negativeMPRK}
				In view of the index function \eqref{eq:indexfun}, the solution-dependent coefficients for MPRK schemes based on RK methods with negative entries in the Butcher tableau not only depend on the step size, solution, and splitting of the right-hand side but also vary with its components. Hence, our formulation \eqref{eq:nsark} actually does not capture this case as we used vector notation. 
				To prove an analogue of Corollary~\ref{cor:orderNSARK} for MPRK schemes based on a Butcher tableau with partially negative entries, we first note that Remark~\ref{rem:compute_u} tells us that using the index function \eqref{eq:indexfun} to switch the PWDs corresponds to switching the colors in the corresponding N-tree. Now, since the condition \eqref{eq:u=1:gamma} needs to be satisfied for all colored  N-trees in $NT_p$, the order conditions for MPRK schemes do not depend on the sign of the Butcher tableau.
			\end{rem}

			\subsection{Order Conditions for MPRK Schemes from the Literature}
			In this subsection we focus on reformulating the order conditions from our theory deriving the sufficient and necessary conditions from the literature, that is the conditions up to order three. 
			Note that, as discussed in Remark \ref{rem:negativeMPRK}, the order conditions for an MPRK scheme do not depend on the sign of the entries of the Butcher tableau. In the following, we thus assume without loss of generality that $\A,\b\geq \bm 0$, so that we can use the representation \eqref{eq:MPRK} of the MPRK scheme. Furthermore, we assume throughout this section that $\Fnu$ is Lipschitz continuous for all $\nu=1,\dotsc,N$.
			
			To formulate the conditions up to the order $p=3$, we observe from the general conditions \eqref{eq:condp=1}, \eqref{eq:condp=2} and \eqref{eq:condp=3} that we should expand $a_{ij}^{[\nu]}(\Y^n,h)$ up to an error of $\O(h^2)$. As we will see, it suffices for our current purposes to assume $\frac{\yi_\nu}{\rho_{i\nu}}=1+\O(h)$ for deriving these expansions.
			\begin{lemma}\label{lem:stageweights}
				Let $a_{ij}^{[\nu]}(\Y^n,h)=a_{ij}\frac{\yi_\nu}{\rho_{i\nu}}$, where $\byi$ is the $i$-th stage of an MPRK method \eqref{eq:MPRK}. Moreover, let $\Fnu\in \mathcal C^{2}$ for $\nu=1,\dotsc,N$, and $\F=\sum_{\nu=1}^N\Fnu$ be the right-hand side of \eqref{ivp}. If \[\frac{\yi_\nu}{\rho_{i\nu}}=1+\O(h), \quad \nu=1,\dots,N,\] then
				\begin{equation*}
					\yi_\nu=y^n_\nu+ hc_iF_\nu(\y^n)+\O(h^2), \quad \nu=1,\dots,N,
				\end{equation*}
				and in particular,
				\begin{equation}\label{eq:aijnu}
					a_{ij}^{[\nu]}(\Y^n,h)=a_{ij} \frac{y^n_\nu+ hc_iF_\nu(\y^n)}{\rho_{i\nu}}+\O(h^2),\quad \nu=1,\dots,N.
				\end{equation}
			\end{lemma}
			\begin{proof}
				The conditions for applying Theorem \ref{thm:main} with $k=1$ are met due to Lemma \ref{lem:aij=O(1)}, so that we can use the expansion of the stages to obtain
				\begin{equation}\label{proofeq:aijnu}
					\yi_\nu=y^n_\nu+ h\sum_{\mu=1}^Nd_i(\rt[]^{[\mu]},\Y^n,h)(\dF(\rt[]^{[\mu]})(\y^n))_\nu+\O(h^2),\quad \nu=1,\dots,N.
				\end{equation}
				Next, we substitute our assumption for the Patankar weights into $a^{[\mu]}_{ij}(\Y^n,h)$ to receive \[a^{[\mu]}_{ij}(\Y^n,h)=a_{ij}\frac{\yi_\mu}{\rho_{i\mu}}=a_{ij}+\O(h), \quad \mu=1,\dots,N\] and find
				\begin{equation*}
					\begin{aligned}
						d_i(\rt[]^{[\mu]},\Y^n,h)&=\sum_{\nu=1}^N\sum_{j=1}^sa^{[\nu]}_{ij}(\Y^n,h)g_j^{[\nu]}(\rt[]^{[\mu]},\Y^n,h)=\sum_{j=1}^sa^{[\mu]}_{ij}(\Y^n,h)=\sum_{j=1}^sa_{ij}+\O(h)=c_i+\O(h).
					\end{aligned}
				\end{equation*}
				Finally, the claim follows after plugging this and $\sum_{\mu=1}^N\dF(\rt[]^{[\mu]})(\y^n)=\sum_{\mu=1}^N\Fmu(\y^n)=\F(\y^n)$ into \eqref{proofeq:aijnu}.
			\end{proof}
			Another helpful result for deriving the known order conditions from the literature is the following.
			\begin{lemma}\label{lem:sigma}
				Let $\A,\b,\c$ describe an explicit $s$-stage Runge--Kutta method of at least order $k$ for some $k\in \N$. Consider the corresponding MPRK scheme \eqref{eq:MPRK} and assume $\F\in \mathcal{C}^{k+1}$. If the MPRK method is of order $k$, then
				\[\sigma_\mu=(\NB_{k-1}(\tfrac1\gamma,\y^n))_\mu+\O(h^{k}),\quad \mu=1,\dotsc, N. \]
			\end{lemma}
			\begin{proof}
				Since the MPRK scheme is of order $k$, we find $\y^{n+1}=\NB_k(\tfrac{1}{\gamma},\y^n)+\O(h^{k+1}).$ According to Lemma~\ref{lem:aij=O(1)} we can apply Corollary~\ref{cor:orderNSARK} to see $ u(\rt[]^{[\mu]},\Y^n,h)=\sum_{j=1}^s b_j\frac{y^{n+1}_\mu}{\sigma_\mu}=1 +\O(h^k)$ for $ \mu=1,\dotsc,N,$ which is equivalent to $\frac{y^{n+1}_\mu}{\sigma_\mu}=1 +\O(h^k)$ for $ \mu=1,\dotsc,N$ as $\sum_{j=1}^sb_j=1$. Using Lemma \ref{lem:aij=O(1)} once again we find $\sigma_\mu=\O(1)$ yielding\[\sigma_\mu=y^{n+1}_\mu+\O(h^k)=(\NB_{k-1}(\tfrac1\gamma,\y^n))_\mu+\O(h^{k}),\quad \mu=1,\dotsc, N.\]
			\end{proof}
			We are now in the position to derive the known order conditions from \cite{MPRK2,MPRK3} for MPRK schemes up to order 3.
			\begin{theorem}\label{thm:MPRKp=1} Let $\Fnu\in \mathcal C^2$ for $\nu=1,\dotsc,N$ and $\A,\b,\c$ describe an explicit $s$-stage Runge--Kutta method of order at least 1.
				The corresponding MPRK schemes \eqref{eq:MPRK} are of order at least 1 if and only if 
				\begin{equation}\label{eq:MPRKcondp=1}
					\sigma_\mu=y^n_\mu+\O(h), \quad \mu=1,\dotsc,N.
				\end{equation}
			\end{theorem}
			\begin{proof}
				The condition \eqref{eq:condp=1} for an MPRK scheme to be at least of order $p=1$ reads
				\begin{equation*}
					\begin{aligned}
						u(\rt[]^{[\mu]},\Y^n,h)=\sum_{j=1}^s b_j\frac{y^{n+1}_\mu}{\sigma_\mu}=1 +\O(h),&& \mu=1,\dotsc,N,
					\end{aligned}
				\end{equation*}
				which can be simplified to $ u(\rt[]^{[\mu]},\Y^n,h)=\frac{y^{n+1}_\mu}{\sigma_\mu}=1+\O(h)$  for $\mu=1,\dotsc,N$ as $\sum_{j=1}^sb_j=1$. From Lemma~\ref{lem:sigma} the condition $\sigma_\mu=y^n_\mu+\O(h)$ for $\mu=1,\dotsc,N$ can be deduced.
				
				Now let \eqref{eq:MPRKcondp=1} be satisfied. It follows from Lemma \ref{lem:aij=O(1)} and \eqref{eq:expyn+1} that $y^{n+1}_\nu=y^n_\nu+\O(h).$ Comparing with \eqref{eq:MPRKcondp=1}, this gives us $u(\rt[]^{[\mu]},\Y^n,h)=\frac{y^{n+1}_\mu}{\sigma_\mu}=1+\O(h)$  for $\mu=1,\dotsc,N$ proving that \eqref{eq:MPRKcondp=1} is sufficient and necessary.
			\end{proof}
			\begin{theorem}\label{thm:MPRKp=2}Let $\Fnu\in \mathcal C^3$ for $\nu=1,\dotsc,N$ and $\A,\b,\c$ describe an explicit $2$-stage Runge--Kutta method of order two.
				Then the corresponding MPRK scheme \eqref{eq:MPRK} is of order two if and only if  
				\begin{subequations}\label{eq:MPRKcondp=2}
					\begin{align}
						\rho_{2\nu}&=y^n_\nu+\O(h), && \nu=1,\dotsc,N,\label{eq:MPRKp=2a}\\
						\sigma_\mu&=(\NB_1(\tfrac{1}{\gamma},\y^n))_\mu+\O(h^2),&& \mu=1,\dotsc,N.\label{eq:MPRKp=2b}
					\end{align}
				\end{subequations}
			\end{theorem}
			\begin{proof}
				First we reduce the necessary and sufficient conditions for $p=2$ from \eqref{eq:condp=2}, which state
				\begin{equation*}
					\begin{aligned}
						u(\rt[]^{[\mu]},\Y^n,h)=\sum_{i=1}^s b_i\frac{y^{n+1}_\mu}{\sigma_\mu}&=1 +\O(h^2),& \mu&=1,\dotsc,N,\\
						u([\rt[]^{[\nu]}]^{[\mu]},\Y^n,h)=\sum_{i,j=1}^s b_i\frac{y^{n+1}_\mu}{\sigma_\mu} a_{ij}\frac{\yi_\nu}{\rho_{i\nu}}&=\frac12 +\O(h),& \mu,\nu&=1,\dotsc,N.\\
					\end{aligned}
				\end{equation*}
				Since $\sum_{i=1}^sb_i=1$, the first equation can be reduced to $\frac{y^{n+1}_\mu}{\sigma_\mu}=1 +\O(h^2)$ for $\mu=1,\dotsc,N$. Plugging this information into the second condition and using $\sum_{j=1}^sa_{ij}=c_i$, we end up with the condition
				$\sum_{i=1}^s b_ic_i\frac{\yi_\nu}{\rho_{i\nu}}=\frac12 +\O(h)$ for $\nu=1,\dotsc,N.$
				Since we assumed $s=2$, we can use $c_1=0$ to obtain the equivalent conditions
				\begin{align*}
					\frac{y^{n+1}_\mu}{\sigma_\mu}&=1 +\O(h^2),\\
					\frac{y^{(2)}_\nu}{\rho_{2\nu}}&=1+\O(h)
				\end{align*}
				for $\mu,\nu=1,\dotsc,N.$
				
				To prove the claim, first assume that the MPRK scheme has order $p=2$. Then Lemma \ref{lem:stageweights} and Lemma~\ref{lem:sigma} yield the conditions from \eqref{eq:MPRKcondp=2}.
				
				Now let \eqref{eq:MPRKcondp=2} be satisfied. Using \eqref{eq:MPRKp=2a} and Lemma \ref{lem:aij=O(1)} together with the expansion \eqref{eq:yithm} of the stages we see $y^{(2)}_\nu=y^n_\nu+\O(h)$, and hence, $\frac{y^{(2)}_\nu}{\rho_{2\nu}}=1+\O(h).$ Moreover, with \eqref{eq:MPRKp=2b} we can apply Theorem~\ref{thm:MPRKp=1} to find that the scheme is at least first order accurate, that is $y^{n+1}_\nu=(\NB_1(\frac{1}{\gamma},\y^n))_\nu+\O(h^2)$. Comparing with \eqref{eq:MPRKp=2b} we end up with $\frac{y^{n+1}_\nu}{\sigma_{\nu}}=1+\O(h^2).$
			\end{proof}
			
			\begin{theorem}\label{thm:MPRKp=3}
				Let $\A,\b,\c$ describe an explicit 3-stage RK scheme and let $\Fnu\in \mathcal C^4$ for $\nu=1,\dotsc,N$. Then the corresponding MPRK scheme \eqref{eq:MPRK} is at least of order $p=s=3$ if and only if  
				\begin{subequations}\label{eq:MPRK3condlit}
					\begin{align}
						\sigma_\mu&=(\NB_2(\tfrac{1}{\gamma},\y^n))_\mu+\O(h^3),&& \mu=1,\dotsc,N,\label{eq:MPRK3condlitc}\\
						\sum_{i=2}^3b_ic_i\frac{y^n_\nu+ hc_iF_\nu(\y^n)}{\rho_{i\nu}}&=\frac12+\O(h^2),&& \nu=1,\dots,N,\label{eq:MPRK3condlitb}\\
						\rho_{i\nu}&=y^n_\nu+\O(h), && \nu=1,\dotsc,N, && i=2,3.\label{eq:MPRK3condlita}       
					\end{align}
				\end{subequations}   
			\end{theorem}
			\begin{proof}
				Similarly as in the proofs before, we obtain the simplified conditions from \eqref{eq:MPRKcondp=3} with $\gamma^{(i)}_\nu=\frac{\yi_\nu}{\rho_{i\nu}}$ and $\gamma^{n+1}_\mu= \frac{y^{n+1}_\mu}{\sigma_\mu}$, which by Lemma \ref{lem:equivalent} are equivalent to 
				\begin{subequations}\label{eq:mprk3}
					\begin{align}
						\frac{y^{n+1}_\mu}{\sigma_\mu}&=1 +\O(h^3), &\mu&=1,\dotsc,N, \label{eq:mprk3a}\\
						\sum_{i=2}^3 b_ic_i\frac{\yi_\nu}{\rho_{i\nu}}&=\frac12 +\O(h^2), &\nu&=1,\dotsc,N,\label{eq:mprk3b}\\
						\frac{\yi_\nu}{\rho_{i\nu}}&=1+\O(h),&\nu&=1,\dotsc,N,\quad i=2,3.\label{eq:mprk3c}
					\end{align}
				\end{subequations} 
				We now show that these conditions are equivalent to \eqref{eq:MPRK3condlit}. First, assuming \eqref{eq:mprk3} is fulfilled, the MPRK scheme is of order 3. Thus, Lemma \ref{lem:sigma} implies \eqref{eq:MPRK3condlitc}. Finally, with \eqref{eq:mprk3c} we are in the position to apply Lemma \ref{lem:stageweights}, which, together with \eqref{eq:mprk3b}, yield the conditions \eqref{eq:MPRK3condlita} and \eqref{eq:MPRK3condlitb}.
				
				Let's now suppose that \eqref{eq:MPRK3condlit} holds. The condition \eqref{eq:mprk3c} follows from \eqref{eq:MPRK3condlita} and the expansion \eqref{eq:yithm} for the stages. Having derived \eqref{eq:mprk3c}, we can apply Lemma \ref{lem:stageweights} to obtain 
				\begin{equation*}
					\yi_\nu=y^n_\nu+ hc_iF_\nu(\y^n)+\O(h^2), \quad \nu=1,\dots,N.
				\end{equation*}
				Together with \eqref{eq:MPRK3condlitb} we can thus conclude \eqref{eq:mprk3b}. Therefore, it remains to deduce condition \eqref{eq:mprk3a}. 
				
				First of all, \eqref{eq:MPRK3condlitc} and Theorem \ref{thm:MPRKp=1} imply that the MPRK scheme is of order at least $1$, which means that $\y^{n+1}=\NB_1(\frac1\gamma,\y^n)+\O(h^2).$ Comparing with \eqref{eq:MPRK3condlitc}, we see \[\frac{y^{n+1}_\mu}{\sigma_\mu}=1 +\O(h^2),\quad \mu=1,\dotsc,N.\] Moreover, since we have already shown \eqref{eq:mprk3c}, we can now verify that condition \eqref{eq:condp=2} is fulfilled which means that the MPRK scheme is even second order accurate. Therefore, we find $\y^{n+1}=\NB_2(\frac1\gamma,\y^n)+\O(h^3)$, so that a comparison with \eqref{eq:MPRK3condlitc} gives us \eqref{eq:mprk3a}.
			\end{proof}
			We have now derived all known order conditions for MPRK schemes from the literature and even proved that they are valid for MPRK schemes based on $\A$ and $\b$ with negative entries.

			\subsection{Reduced Order Conditions for 4th Order MPRK Methods}
			The main idea in deriving the known conditions for 3rd order MPRK schemes was to use Lemma \ref{lem:equivalent} for reducing the order conditions \eqref{eq:condp=3} and then substituting the expansions for the stages to obtain conditions depending only on $\y^n$ and $h$. 
			
			Similarly, we are in the position to derive conditions for 4th order by first using Lemma \ref{lem:equivalent4} to reduce the order conditions \eqref{eq:condp=4}. Now, in order to eliminate the dependency of the conditions on the stages, we need to expand $a_{ij}^{[\nu]}(\Y^n,h)$ up to an error of $\O(h^3)$ giving an analogue to Lemma \ref{lem:stageweights}. However, since equation \eqref{eq:MPRKcondp'=4} in Lemma \ref{lem:equivalent4} gives us $\gamma^{(i)}_\nu=1+\O(h^2)$, we will see that it suffices to prove the following lemma assuming $\frac{\yi_\nu}{\rho_{i\nu}}=1+\O(h^2)$.
			\begin{lemma}\label{lem:stageweights4}
				Let $a_{ij}^{[\nu]}(\Y^n,h)=a_{ij}\frac{\yi_\nu}{\rho_{i\nu}}$, where $\byi$ is the $i$-th stage of an MPRK method \eqref{eq:MPRK}. Moreover, let $\Fnu\in \mathcal C^{3}$ for $\nu=1,\dotsc,N$, and $\F=\sum_{\nu=1}^N\Fnu$ be the right-hand side of \eqref{ivp}. If \[\frac{\yi_\nu}{\rho_{i\nu}}=1+\O(h^2), \quad \nu=1,\dots,N,\] then
				\begin{equation*}
					\yi_\nu=y^n_\nu+ hc_iF_\nu(\y^n)+\frac12h^2\sum_{k=1}^sa_{ik}c_k(\D\F(\y^n)\F(\y^n))_\nu+\O(h^3), \quad \nu=1,\dots,N,
				\end{equation*}
				and in particular,
				\begin{equation}\label{eq:aijnu4}
					a_{ij}^{[\nu]}(\Y^n,h)=a_{ij} \frac{y^n_\nu+ hc_iF_\nu(\y^n)+\frac12h^2\sum_{k=1}^sa_{ik}c_k(\D\F(\y^n)\F(\y^n))_\nu}{\rho_{i\nu}}+\O(h^3),\quad \nu=1,\dots,N.
				\end{equation}
			\end{lemma}
			\begin{proof}
				The conditions for applying Theorem \ref{thm:main} with $k=2$ are met due to Lemma \ref{lem:aij=O(1)}, so that we can use the expansion of the stages to obtain 
				\begin{equation}\label{proofeq:aijnu4}
					\yi_\nu=y^n_\nu+ h\sum_{\mu=1}^Nd_i(\rt[]^{[\mu]},\Y^n,h)(\dF(\rt[]^{[\mu]})(\y^n))_\nu+\frac12h^2\sum_{\mu,\eta=1}^Nd_i\left(\left[\rt[]^{[\eta]}\right]^{[\mu]},\Y^n,h\right)\left(\dF\left(\left[\rt[]^{[\eta]}\right]^{[\mu]}\right)(\y^n)\right)_\nu+\O(h^3)
				\end{equation}
				for $\nu=1,\dots,N$.
				Moreover, we know that
				\begin{equation}\label{eq:proofaij}
					a^{[\mu]}_{ij}(\Y^n,h)=a_{ij}\frac{\yi_\mu}{\rho_{i\mu}}=a_{ij}+\O(h^2), \quad \mu=1,\dots,N
				\end{equation}
				as well as 
				\begin{equation}\label{eq:proofdi}
					d_i(\rt[]^{[\mu]},\Y^n,h)=\sum_{j=1}^sa^{[\mu]}_{ij}(\Y^n,h)=\sum_{j=1}^sa_{ij}+\O(h^2)=c_i+\O(h^2)
				\end{equation}
				by following the lines of the proof of Lemma \ref{lem:stageweights}. Moreover, in that proof we have already seen that $\sum_{\mu=1}^N\dF(\rt[]^{[\mu]})(\y^n)=\F(\y^n)$ so that we obtain the intermediate result
				\begin{equation*}
					\yi_\nu=y^n_\nu+ hc_iF_\nu(\y^n)+\frac12h^2\sum_{\mu,\eta=1}^Nd_i\left(\left[\rt[]^{[\eta]}\right]^{[\mu]},\Y^n,h\right)\left(\dF\left(\left[\rt[]^{[\eta]}\right]^{[\mu]}\right)(\y^n)\right)_\nu+\O(h^3).
				\end{equation*} 
				Turning to the coefficient of $h^2$, we first point out that, according to \eqref{eq:pertcond}, we have
				\begin{equation*}
					\begin{aligned}						d_i\left(\left[\rt[]^{[\eta]}\right]^{[\mu]},\Y^n,h\right)&=\sum_{\nu=1}^N\sum_{j=1}^sa^{[\nu]}_{ij}(\Y^n,h)g_j^{[\nu]}\left(\left[\rt[]^{[\eta]}\right]^{[\mu]},\Y^n,h\right)=\sum_{j=1}^sa^{[\mu]}_{ij}(\Y^n,h)d_j\left(\rt[]^{[\eta]},\Y^n,h\right)\\
						&=\sum_{j=1}^sa_{ij}c_j+\O(h^2),
					\end{aligned}
				\end{equation*}
				where we used \eqref{eq:proofaij} and \eqref{eq:proofdi}.
				Finally, using \eqref{eq:elemdiff} we find
				\[\sum_{\mu,\eta=1}^N\dF\left(\left[\rt[]^{[\eta]}\right]^{[\mu]}\right)(\y^n)=\sum_{\mu,\eta=1}^N\sum_{i_1=1}^d\partial_{i_1}\Fmu(\y^n)\dF_{i_1}(\rt[]^{[\eta]})(\y)=\sum_{\mu=1}^N\D\Fmu(\y^n)\sum_{\eta=1}^N\F^{[\eta]}(\y^n)=\D\F(\y^n)\F(\y^n).\]
				The claim follows after substituting these equations into \eqref{proofeq:aijnu4}.
			\end{proof}
			With that lemma we now derive conditions for 4th order MPRK schemes.
			\begin{theorem}\label{thm:MPRKp=4}
				Let $\A,\b,\c$ describe an explicit 4-stage RK scheme of order 4 and let $\Fnu\in \mathcal C^5$ for $\nu=1,\dotsc,N$. Then the corresponding MPRK scheme \eqref{eq:MPRK} is at least of order $p=s=4$ if and only if  for $\mu,\nu=1,\dotsc,N$ and $i=2,3,4$ we have
				\begin{subequations}\label{eq:MPRK4condlit}
					\begin{align}
						\sigma_\mu&=(\NB_3(\tfrac{1}{\gamma},\y^n))_\mu+\O(h^4),\label{eq:MPRK4condlitc}\\
						\sum_{i=2}^4b_ic_i\frac{y^n_\nu+ hc_iF_\nu(\y^n)+\frac12h^2\sum_{k=1}^4a_{ik}c_k(\D\F(\y^n)\F(\y^n))_\nu}{\rho_{i\nu}}&=\frac12+\O(h^3),\label{eq:MPRK4condlitb}\\
						\rho_{i\nu}&=y^n_\nu+ hc_iF_\nu(\y^n)+\O(h^2).\label{eq:MPRK4condlita}       
					\end{align}
				\end{subequations}   
			\end{theorem}
			\begin{proof}
				The simplified conditions for an MPRK method of order 4 are given by \eqref{eq:MPRKcondp=4} with $\gamma^{(i)}_\nu=\frac{\yi_\nu}{\rho_{i\nu}}$ and $\gamma^{n+1}_\mu= \frac{y^{n+1}_\mu}{\sigma_\mu}$. Using Lemma \ref{lem:equivalent4} these conditions are equivalent to 
				\begin{subequations}\label{eq:mprk4}
					\begin{align}
						\frac{y^{n+1}_\mu}{\sigma_\mu}&=1 +\O(h^4), &\mu&=1,\dotsc,N, \label{eq:mprk4a}\\
						\sum_{i=2}^4 b_ic_i\frac{\yi_\nu}{\rho_{i\nu}}&=\frac12 +\O(h^3), &\nu&=1,\dotsc,N,\label{eq:mprk4b}\\
						\frac{\yi_\nu}{\rho_{i\nu}}&=1+\O(h^2),&\nu&=1,\dotsc,N,\quad i=2,3,4.\label{eq:mprk4c}
					\end{align}
				\end{subequations} 
				We now show that the conditions \eqref{eq:MPRK4condlit} and \eqref{eq:mprk4} are equivalent. We start by assuming that \eqref{eq:mprk4} is fulfilled and note that this part works along the same lines as in Theorem \ref{thm:MPRKp=3}. Nevertheless, we present it here for the sake of completeness. 
				
				Now, since \eqref{eq:mprk4} holds, the MPRK scheme is of order 4. Thus, Lemma \ref{lem:sigma} implies \eqref{eq:MPRK4condlitc}. Finally, with \eqref{eq:mprk4c} we are in the position to apply Lemma \ref{lem:stageweights4}, which, together with \eqref{eq:mprk4b}, yield the conditions \eqref{eq:MPRK4condlita} and \eqref{eq:MPRK4condlitb}.
				
				Let's now suppose that \eqref{eq:MPRK4condlit} holds. Using the Taylor expansion \eqref{eq:yithm} for the stages, we first observe with \eqref{eq:MPRK4condlita} that $\frac{\yi_\nu}{\rho_{i\nu}}=1+\O(h)$. Applying Lemma \ref{lem:stageweights}, we see $\yi_\nu=y^n_\nu+hc_iF_\nu(\y^n)+\O(h^2)$, and thus, comparing with \eqref{eq:MPRK4condlita}, we derived \eqref{eq:mprk4c}. As a result, we can now apply Lemma \ref{lem:stageweights4} to obtain 
				\begin{equation}\label{eq:stages4thorder}
					\yi_\nu=y^n_\nu+ hc_iF_\nu(\y^n)+\frac12h^2\sum_{k=1}^4a_{ik}c_k(\D\F(\y^n)\F(\y^n))_\nu+\O(h^3), \quad \nu=1,\dots,N.
				\end{equation}
				As a direct consequence of this and \eqref{eq:MPRK4condlitb}, we thus  conclude \eqref{eq:mprk4b}. Therefore, it remains to deduce condition \eqref{eq:mprk4a}. First of all, \eqref{eq:MPRK4condlitc} and Theorem \ref{thm:MPRKp=1} imply that the MPRK scheme is of order at least $1$, i.\,e.\ $\y^{n+1}=\NB_1(\frac1\gamma,\y^n)+\O(h^2).$ Comparing with \eqref{eq:MPRK4condlitc}, we see \[\frac{y^{n+1}_\mu}{\sigma_\mu}=1 +\O(h^2),\quad \mu=1,\dotsc,N.\] Moreover, since we have already shown \eqref{eq:mprk4c}, we can now verify that condition \eqref{eq:condp=2} is fulfilled which means that the MPRK scheme is even second order accurate. Therefore, we find $\y^{n+1}=\NB_2(\frac1\gamma,\y^n)+\O(h^3)$, so that a comparison with \eqref{eq:MPRK4condlitc} gives us 
				\[\frac{y^{n+1}_\mu}{\sigma_\mu}=1 +\O(h^3),\quad \mu=1,\dotsc,N.\] Finally, using this and \eqref{eq:mprk4c} once again, we even fulfill the conditions \eqref{eq:condp=3} proving the 3rd order accuracy of the scheme, that is $\y^{n+1}=\NB_3(\frac1\gamma,\y^n)+\O(h^4)$. Comparing a last time with \eqref{eq:MPRK4condlitc} gives us \eqref{eq:mprk4a}.
			\end{proof}
			With this proof, we obtain for the first time order conditions
			for 4th order MPRK methods. We have even proved that these conditions are necessary and sufficient.
			A first intuitive, yet rather expensive way of achieving $4$th order would be to use lower order MPRK methods for the computation of the PWDs. In particular, we propose the following method based on the classical Runge--Kutta method described by the Butcher tableau
			\begin{equation*}
				\begin{aligned}
					\def\arraystretch{1.2}
					\begin{array}{c|cccc}
						0 &  & & & \\
						\frac12 & \frac12 & & &\\
						\frac12 &0 &\frac12 & &\\
						1 &0 & 0 &1 &\\
						\hline
						& \frac16 &\frac13&\frac13 & \frac16
					\end{array}
				\end{aligned}
			\end{equation*}
			as a proof of concept scheme. We know that the PWD $\bm\sigma$ needs to be a third order approximation to $\y^{n+1}$, for which we use the MPRK43($0.5,0.75$) method derived in \cite{MPRK3}. Within this method, there is a second order scheme embedded, which we denote by $\hat{\bm\sigma}$ and use to compute $\bm \rho_i$ for $i=2,3,4$ using $c_i\dt$ as a time step, resulting in $\frac{\yi_\nu}{\rho_{i\nu}}=1+\O(\dt^3)$, see \eqref{eq:stages4thorder}. Now, according to Corollary~\ref{cor:suff_cond} the overall method is of order 4. 
			
			The third order scheme  returns $\hat{\y}$ and consists of solving $4$ linear systems, where $\hat{\bm\sigma}$ requires solving $2$ systems. However, as $c_4=1$, we can actually use $\bm\rho_{4}=\hat{\bm\sigma}$, and since $c_2=c_3$ we can use $\bm \rho_{2}=\bm \rho_{3}$. Finally, the modified Patankar trick (MP) applied to the classical RK method also adds $4$ linear systems to our list.  Altogether $\bm \rho_{2}$ and $\bm \rho_{3}$ yield a total of 2 linear systems, $\bm \rho_{4}=\hat{\bm\sigma}$ and $\bm\sigma=\hat{\y}$ need the solution of $4$ linear systems and the MP approach applied to the classical RK scheme results in another $4$ linear systems giving us a total of $10$ stages and linear systems to solve. The minimal amount of linear systems of course would be $4$ and to reduce the number of linear systems to be solved will be part of our future work. An indication that this is possible is given by MPDeC methods \cite{MPDeC} where fourth order is obtained by $7$ stages for Gauss--Lobatto nodes. Nevertheless, our first attempt has as many stages as the fourth order MPDeC scheme based on equispaced nodes. 
			Altogether, we obtain the following fourth order MPRK scheme
			\begin{equation*}
				\begin{aligned}
					y^{(1)}_k &= y^n_k,\nonumber\\
					y^{(2)}_k &= y^n_k
					+  \frac12 \dt \sum_{\nu=1}^N\left(
					p_{k\nu}\bigl( \y^n \bigr) \frac{y^{(2)}_\nu}{\rho_{2\nu}}
					- d_{k\nu}\bigl( \y^n \bigr) \frac{y^{(2)}_k}{\rho_{2k}}
					\right),
					\nonumber\\
					y^{(3)}_k &= y^n_k
					+\frac12 \dt \sum_{\nu=1}^N
					\Biggl(p_{k\nu} \bigl(\y^{(2)}\bigr)  \frac{  y_\nu^{(3)}
					}{\rho_{3\nu} }						- d_{k\nu} \bigl(\y^{(2)}\bigr)  \frac{  y_k^{(3)}
					}{\rho_{3k}}
					\Biggr),\\
					y^{(4)}_k &= y^n_k
					+\dt \sum_{\nu=1}^N
					\Biggl(p_{k\nu} \bigl(\y^{(3)}\bigr)  \frac{  y_\nu^{(4)}
					}{\rho_{4\nu} }						- d_{k\nu} \bigl(\y^{(3)}\bigr)  \frac{  y_k^{(4)}
					}{\rho_{4k}}
					\Biggr),\\
					y^{n+1}_k &= y^n_k
					+ \frac16 \dt \sum_{\nu=1}^N \Biggl(
					\left( p_{k\nu}\bigl( \y^n \bigr) +2 p_{k\nu}\bigl( \y^{(2)} \bigr)
					+ 2 p_{k\nu}\bigl( \y^{(3)} \bigr) +p_{k\nu}\bigl( \y^{(4)} \bigr)
					\right) \frac{y^{n+1}_\nu}{\sigma_\nu}
					\nonumber\\&\qquad\qquad\qquad
					- \left( d_{k\nu}\bigl( \y^n \bigr) +2 d_{k\nu}\bigl( \y^{(2)} \bigr)
					+ 2 d_{k\nu}\bigl( \y^{(3)} \bigr) + d_{k\nu}\bigl( \y^{(4)} \bigr)
					\right) \frac{y^{n+1}_k}{\sigma_k}
					\Biggr),
				\end{aligned}
			\end{equation*}
			where the PWDs are computed using the MPRK43(0.5,0.75) method and the remarks above. In particular 
			\begin{equation*}
				\begin{aligned}
					\hat y^{(1)}_k &= y^n_k,\\
					\hat y^{(2)}_k &= y^n_k
					+  \frac12 \dt \sum_{\nu=1}^N\left(
					p_{k\nu}\bigl( \y^n \bigr) \frac{\hat y^{(2)}_\nu}{y^n_\nu}
					- d_{k\nu}\bigl( \y^n \bigr) \frac{\hat y^{(2)}_k}{y^n_k}
					\right),
					\\
					\hat \sigma_k &= y^n_k
					+\dt \sum_{\nu=1}^N
					\Biggl(p_{k\nu} \bigl(\hat \y^{(2)}\bigr)  \frac{  \hat \sigma_\nu
					}{\bigl(\hat y_\nu^{(2)} \bigr)^{2}
						\bigl(y_\nu^n\bigr)^{-1} }						- d_{k\nu} \bigl(\hat \y^{(2)}\bigr)  \frac{ \hat \sigma_k}{\bigl(\hat y_k^{(2)} \bigr)^{2}
						\bigl(y_k^n\bigr)^{-1} }
					\Biggr)
				\end{aligned}
			\end{equation*}
			is the second order embedded scheme and gives us a function $\hat\sigma_\nu=\hat\sigma_\nu(h)$. As described above we then set $\rho_{2\nu}\coloneqq\rho_{3\nu}\coloneqq \hat\sigma_\nu(\frac12 h)$ and $\rho_{4\nu}\coloneqq\hat\sigma_\nu(h)$. Additionally, the remaining stages of the MPRK43(0.5,0.75) scheme are used to compute $\sigma_k\coloneqq \hat y^{n+1}_k$ according to
			\begin{equation*}
				\begin{aligned}
					\hat y^{(3)}_k &= y^n_k
					+\frac34 \dt \sum_{\nu=1}^N
					\Biggl(p_{k\nu} \bigl(\hat \y^{(2)}\bigr)  \frac{  \hat y_\nu^{(3)}
					}{\bigl(\hat y_\nu^{(2)} \bigr)^{2}
						\bigl(y_\nu^n\bigr)^{-1} }						- d_{k\nu} \bigl(\hat \y^{(2)}\bigr)  \frac{  \hat y_k^{(3)}
					}{\bigl(\hat y_k^{(2)} \bigr)^{2}
						\bigl(y_k^n\bigr)^{-1} }
					\Biggr),\\
					\hat y^{n+1}_k &= y^n_k
					+ \frac19 \dt \sum_{\nu=1}^N \Biggl(
					\left( 2p_{k\nu}\bigl( \hat \y^n \bigr) +3 p_{k\nu}\bigl( \hat \y^{(2)} \bigr)
					+ 4 p_{k\nu}\bigl( \hat \y^{(3)} \bigr) \bigr)
					\right) \frac{\hat y^{n+1}_\nu}{\hat \sigma_\nu}
					\nonumber\\&\qquad\qquad\qquad
					- \left( 2d_{k\nu}\bigl( \hat \y^n \bigr) +3 d_{k\nu}\bigl(\hat  \y^{(2)} \bigr)
					+ 4 d_{k\nu}\bigl( \hat \y^{(3)} \bigr)  \bigr)
					\right) \frac{\hat y^{n+1}_k}{\hat \sigma_k}
					\Biggr).
				\end{aligned}
			\end{equation*}
			We note that, in principle, any member of the two MPRK43 families can be used resulting in two families of 4th order MPRK schemes.
			
			The experimental order of convergence of our new fourth order MPRK method, denoted by MPRKord4, is verified in Figure~\ref{fig:MPRK4EOC}, where besides the linear system 
			\begin{equation}\label{eq:EOC_testprob}
				\y'(t)=\vec{-5 & \hphantom{-}1\\ \hphantom{-}5& -1}\y(t), \quad \y(0)=\vec{0.9\\0.1}
			\end{equation}
			as suggested in \cite{MPRK3}, we solved also the NPZD problem \cite{BDM2005}, the Brusselator problem \cite{LefeverNicolis1971,HNW1993}, and the SIR problem \cite{MW2013}. We plot the error of the numerical solution at the final time, where the reference solution was computed with the \MATLAB ODE solver \code{ode45} using \code{RelTol = AbsTol = 1e-13}.
			\begin{figure}[!htbp]
				\begin{subfigure}[t]{0.495\textwidth}
					\includegraphics[width=\textwidth]{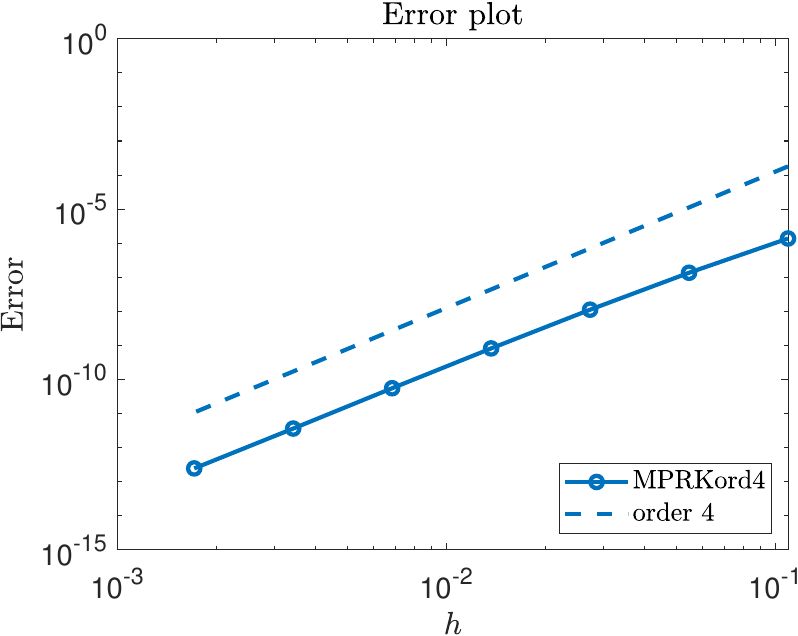}
					\caption{Linear problem \eqref{eq:EOC_testprob} with final time $\tend=1.75$}
				\end{subfigure}
				\begin{subfigure}[t]{0.5\textwidth}
					\includegraphics[width=\textwidth]{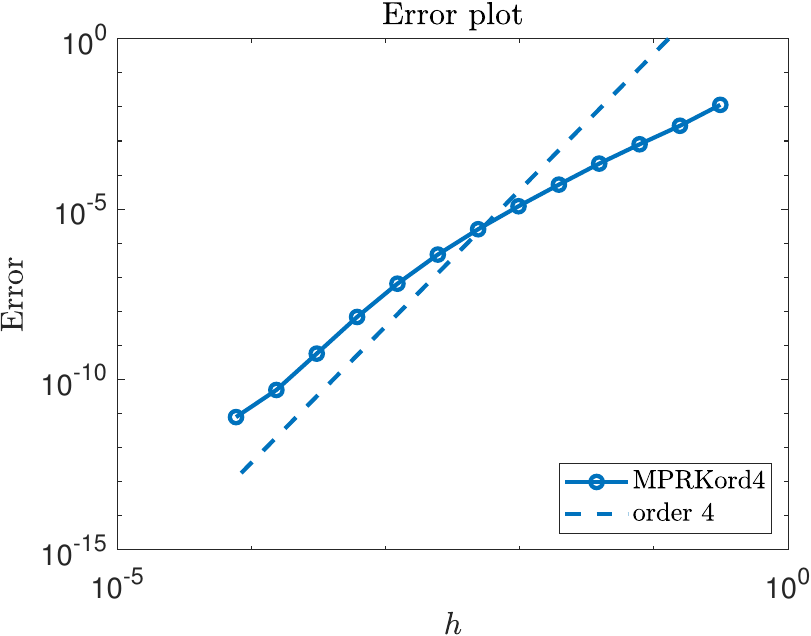}
					\caption{NPZD problem \cite{BDM2005} with final time $\tend=5$}
				\end{subfigure}\\
				\begin{subfigure}[t]{0.5\textwidth}
					\includegraphics[width=\textwidth]{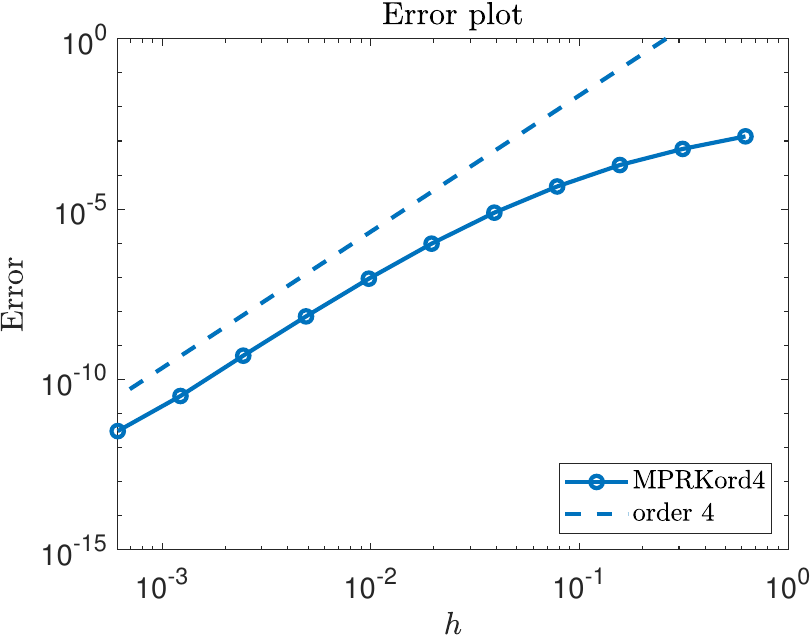}
					\caption{Brusselator problem \cite{LefeverNicolis1971,HNW1993} with final time $\tend=10$}
				\end{subfigure}
				\begin{subfigure}[t]{0.5\textwidth}
					\includegraphics[width=\textwidth]{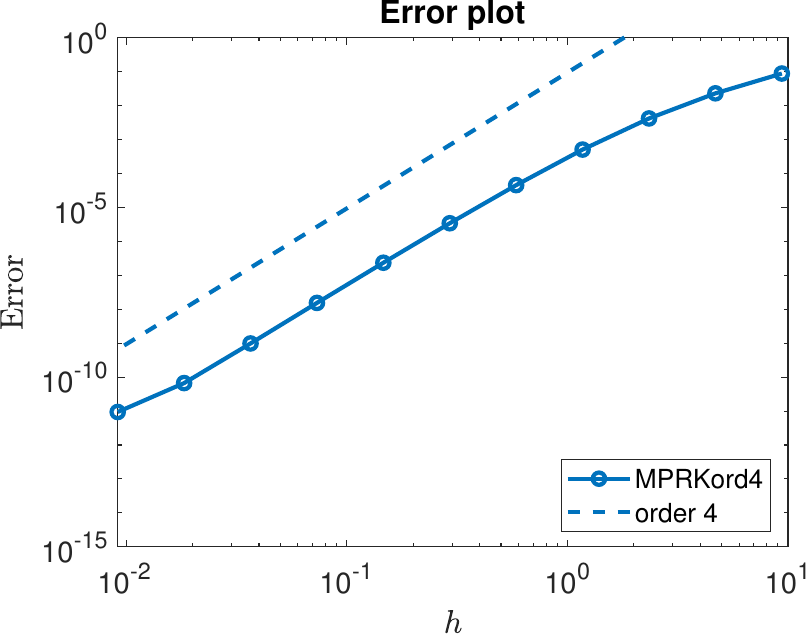}
					\caption{SIR problem \cite{MW2013} with final time $\tend=150$}
				\end{subfigure}
				\caption{Error plot of MPRKord4 applied to various problems. The error was computed at the respective final times using \code{ode45} as a reference solution.}\label{fig:MPRK4EOC}
			\end{figure}
			
			\section{Summary and Outlook}
			In this work, we have presented a comprehensive theory for the analysis of accuracy of non-standard additive Runge--Kutta methods based on the theory of NB-series. As two popular examples, we chose to derive the known order conditions for modified Patankar--Runge--Kutta (MPRK) methods as well as Geometric Conservative (GeCo) schemes. Furthermore, based on the presented theory, the construction of new and higher order MPRK and GeCo methods is possible. In the case of MPRK schemes, these order conditions in general depend on also on the stages, however, with further effort, this dependency can be eliminated. Using the presented theory, for the first time conditions for 3rd and 4th order GeCo schemes as well as 4th order MPRK methods are available. In addition, we reduced these conditions to simpler but equivalent conditions within this work and eliminated the stage-dependency in the case of MPRK methods. Finally, we constructed a new MPRK scheme of order four using the presented theory and numerically confirmed the expected order of convergence. 
			
			Future research topics include the investigation of the order for stiff problems, for instance following \cite{RS2022}, the construction of Patankar-weight denominators $\rho_{i\nu}$ and $\sigma_i$ as well as denominator functions $\phi_i$ and $\phi_{n+1}$ to fulfill the order conditions for 4th order MPRK with fewer stages as well as 3rd and 4th order GeCo methods, respectively. Furthermore, deriving order conditions for other classes of methods such as semi-implicit Runge--Kutta methods \cite{SIRK23} or modified Patankar SSPRK schemes \cite{SSPMPRK2,SSPMPRK3} is of interest.  
			\section*{Acknowledgements}
			The author T.\ Izgin gratefully acknowledges the financial support by the Deutsche Forschungsgemeinschaft (DFG, project number 466355003) through the grant ME 1889/10-1.
			\appendix
			\section{Results for the Derivation of the Main Theorem}
			In this appendix we present and prove intermediate results that are analogous to statements in \cite{B16}.
			We start by recalling Theorem 308A from \cite{B16}, for which we briefly introduce the notation.
			
			Let $m\in \N$ and $I$ be a non-decreasing and finite sequence of integers from the set $\{1,2,\dotsc,m\}$ and $J_m$ the set of all such $I$, where we also include the empty sequence $\varnothing\in J_m$. If $I$ contains $k_j$ occurrences of $j$ for each $j=1,\dotsc, m$ then we define 
			\[\hat\sigma(I)=\prod_{j=1}^mk_j!\]
			and set $\hat\sigma(\varnothing)=1$.
			Now let $\bd^{(1)},\dotsc,\bd^{(m)}\in \R^d$ and define for $I=(i_1,\dotsc,i_l)\in J_m $ the quantity $\lvert I\rvert=l$ as well as
			\[\bd^{I}=(\bd^{(i_1)},\dotsc,\bd^{(i_l)})\in (\R^d)^l,\] and we set $\bd^{\varnothing}=\varnothing$ as well as $\lvert \varnothing\rvert=0.$ Next, for a map $\f\in \mathcal C^{k+1}( \R^d, \R^d)$ we define $\f^{(0)}(\y)\varnothing=\f(\y)$ and \[\f^{(l)}(\y)\bd^I=\sum_{j_1,\dotsc,j_l=1}^d\partial_{j_1,\dotsc,j_l}\f(\y)\delta^{(i_1)}_{j_1}\cdots\delta^{(i_l)}_{j_l},\quad 1\leq l\leq k+1,\]
			which allows us to formulate \cite[Theorem 308A]{B16}, where we truncate the series using the Lagrangian remainder.
			\begin{theorem}\label{thm:308a}
				Let $p\in \N$ and $f\in \mathcal C^{p+1}(\R^d,\R)$ as well as $\y,\bm{\delta}^{(1)},\dotsc,\bm{\delta}^{(m)}\in \R^d$. Then
				\[ f\left(\y +\sum_{i=1}^m\bm{\delta}^{(i)}\right)=\sum_{\substack{I\in J_m\\ \lvert I\rvert \leq p}}\frac{1}{\hat\sigma(I)} f^{(\lvert I\rvert)}(\y)\bm{\delta}^I+  R_p\left(\y+\sum_{i=1}^m\bm{\delta}^{(i)},\sum_{i=1}^m\bm{\delta}^{(i)}\right),\]
				where, using the multi index notation, we have
				\[R_p\left(\mathbf x,\mathbf a\right)=\sum_{\lvert \bm \alpha \rvert =p+1}\frac{\partial^{\bm \alpha} f(\bm \xi)}{\bm \alpha!}(\mathbf x-\mathbf a)^{\bm\alpha}\]
				with $\bm{\alpha}\in \N_0^d$ and $\xi_j$ between $x_j$ and $a_j$. 
			\end{theorem}
			The key observation is that this equality holds true for any values of $\bd^{(i)}$, that is also for solution-dependent vectors $\bd^{(i)}=\bd^{(i)}(\Y^n,h)$.
			
			Our aim is to apply Theorem \ref{thm:308a} to each addend of the right-hand side of the differential equations \eqref{ivp}. Following the idea from \cite[Lemma 310B]{B16}, we prove the following result.
			\begin{lemma}\label{lem:hFl}
				Let $k\in \N$ and $\Fnu\in \mathcal C^{k+1}$ for $\nu=1\dotsc,N$. Then 
				\[h\Fnu\left(\y^n+\sum_{\tau\in NT_{k-1}}\theta(\tau,\Y^n,h)\frac{h^{\lvert\tau\rvert}}{\sigma(\tau)}\dF(\tau)(\y^n) +\O(h^k)\right)=\sum_{\tau\in NT_k}\widetilde\theta_\nu(\tau,\Y^n,h)\frac{h^{\lvert\tau\rvert}}{\sigma(\tau)}\dF(\tau)(\y^n)+\O(h^{k+1}),\]
				where
				\begin{equation}\label{eq:widetildetheta}
					\widetilde\theta_\nu(\tau,\Y^n,h)=\begin{cases}
						\delta_{\nu\mu},& \tau=\rt[]^{[\mu]},\\
						\delta_{\nu\mu}\prod_{i=1}^l\theta(\tau_i,\Y^n,h),& \tau=[\tau_1,\dotsc,\tau_l]^{[\mu]}
					\end{cases}
				\end{equation}
				and $\delta_{\nu\mu}$ denotes the Kronecker delta.
			\end{lemma}
			\begin{proof} Let $NT_{k-1}=\{\tau^{(i)}\mid i=1,\dotsc,\lvert NT_{k-1}\rvert\}$.
				We want to apply Theorem \ref{thm:308a} by first writing 
				\begin{equation*}
					\begin{aligned}
						h\Fnu\left(\y^n+\sum_{\tau\in NT_{k-1}}\theta(\tau,\Y^n,h)\frac{h^{\lvert\tau\rvert}}{\sigma(\tau)}\dF(\tau)(\y^n) +\O(h^k)\right)&=h\Fnu\left(\y^n+\sum_{j=1}^{\lvert NT_{k-1}\rvert}\bd^{(j)}+\O(h^k)\right)\\
						&=h\Fnu\left(\y^n+\sum_{j=1}^{\lvert NT_{k-1}\rvert}\bd^{(j)}\right)+\O(h^{k+1})
					\end{aligned}
				\end{equation*}
				with
				\begin{equation}
					\bd^{(j)}=\theta(\tau^{(j)},\Y^n,h)\frac{h^{\lvert \tau^{(j)}\rvert}}{\sigma(\tau^{(j)})}\dF(\tau^{(j)})(\y^n). \label{eq:delta(i)}
				\end{equation}
				To that end, we first introduce for a sequence $I=(i_1,\dotsc,i_r)\in J_m$ with $m=\lvert NT_{k-1}\rvert$ the quantity $\Sigma_I=\sum_{j=1}^r\lvert\tau^{(i_j)}\rvert$ and set $\Sigma_\varnothing=0$. With that, Theorem \ref{thm:308a} and \eqref{eq:delta(i)} yield
				\[ h\Fnu\left(\y^n+\sum_{\tau\in NT_{k-1}}\theta(\tau,\Y^n,h)\frac{h^{\lvert\tau\rvert}}{\sigma(\tau)}\dF(\tau)(\y^n) +\O(h^k)\right)=\sum_{\substack{I\in J_m\\ \Sigma_I\leq k-1}}\frac{h}{\hat\sigma(I)}(\Fnu)^{(\lvert I\rvert)}(\y^n)\bd^I +\O(h^{k+1}).\]
				To prove the claim, we show that 
				\[\sum_{\substack{I\in J_m\\ \Sigma_I\leq k-1}}\frac{h}{\hat\sigma(I)}(\Fnu)^{(\lvert I\rvert)}(\y^n)\bd^I =\sum_{\tau\in NT_k}\widetilde\theta_\nu(\tau,\Y^n,h)\frac{h^{\lvert\tau\rvert}}{\sigma(\tau)}\dF(\tau)(\y^n)\]
				by means of an induction.
				
				If $k=1$, we find $\frac{h}{\hat\sigma(\varnothing)}(\Fnu)^{(0)}(\y^n)\varnothing=h\Fnu(\y^n)$ and 
				\[\sum_{\mu=1}^N\widetilde\theta_\nu(\rt[]^{[\mu]},\Y^n,h)\frac{h^{\lvert\tau\rvert}}{\sigma(\rt[]^{[\mu]})}\dF(\rt[]^{[\mu]})(\y^n)=\sum_{\mu=1}^N\delta_{\nu\mu}h\Fmu(\y^n)=h\Fnu(\y^n),\]
				so that
				\begin{equation*}
					\begin{aligned}
						\sum_{\substack{I\in J_m\\ \Sigma_I\leq 0}}\frac{h}{\hat\sigma(I)}(\Fnu)^{(\lvert I\rvert)}(\y^n)\bd^I&=\frac{h}{\hat\sigma(\varnothing)}(\Fnu)^{(0)}(\y^n)\varnothing =h\Fnu(\y^n)=\sum_{\mu=1}^N\widetilde\theta_\nu(\rt[]^{[\mu]},\Y^n,h)\frac{h^{\lvert\tau\rvert}}{\sigma(\rt[]^{[\mu]})}\dF(\rt[]^{[\mu]})(\y^n)\\& = \sum_{\tau\in NT_1}\widetilde\theta(\tau,\Y^n,h)\frac{h^{\lvert\tau\rvert}}{\sigma(\tau)}\dF(\tau)(\y^n) 
					\end{aligned}
				\end{equation*} is true.
				By induction we can now assume that 
				\[\sum_{\substack{I\in J_m\\ \Sigma_I\leq k-2}}\frac{h}{\hat\sigma(I)}(\Fnu)^{(\lvert I\rvert)}(\y^n)\bd^I =\sum_{\tau\in NT_{k-1}}\widetilde\theta_\nu(\tau,\Y^n,h)\frac{h^{\lvert\tau\rvert}}{\sigma(\tau)}\dF(\tau)(\y^n)\]
				holds true for some $k\geq 2$, so that it remains to show
				\begin{equation}\label{eq:induction}
					\sum_{\substack{I\in J_m\\ \Sigma_I=k-1}}\frac{h}{\hat\sigma(I)}(\Fnu)^{(\lvert I\rvert)}(\y^n)\bd^I =\sum_{\tau\in NT_k\setminus NT_{k-1}}\widetilde\theta_\nu(\tau,\Y^n,h)\frac{h^{\lvert\tau\rvert}}{\sigma(\tau)}\dF(\tau)(\y^n)
				\end{equation}
				to finish the proof by induction. For this, let us consider an arbitrary element $\tau\in NT_k\setminus NT_{k-1}$, which can be written as
				\begin{equation}\label{eq:taulem}
					\tau=[\tau_1,\dotsc,\tau_l]^{[\mu]}=[(\tau^{(i_1)})^{m_1},\dotsc,(\tau^{(i_j)})^{m_j}]^{[\mu]}, \quad \sum_{r=1}^j\lvert\tau^{(i_r)}\rvert m_r=k-1\footnote{The root node is not counted here.},
				\end{equation}
				where we point out that $\tau^{(i_r)}\in NT_{k-1}$ for $r=1,\dotsc,j$. Without loss of generality, we can assume that $i_1<i_2<\dotsc <i_j$. 
				For each such $\tau$ we can define the uniquely determined and non-decreasing sequence \[ \hat I=(\underbrace{i_1,\dotsc,i_1}_{m_1 \text{times}},\underbrace{i_2,\dotsc,i_2}_{m_2 \text{times}},\dotsc,\underbrace{i_j,\dotsc,i_j}_{m_j \text{times}})\]
				satisfying $\hat I\in J_m$ and $\Sigma_{\hat I}=\sum_{r=1}^l\lvert \tau_r \rvert=\sum_{r=1}^j\lvert\tau^{(i_r)}\rvert m_r=k-1$, so that equation \eqref{eq:induction} follows by proving
				\[\frac{h}{\hat\sigma(\hat I)}(\Fnu)^{(\lvert \hat I\rvert)}(\y^n)\bd^{\hat I} =\sum_{\mu=1}^N\widetilde\theta_\nu([\tau_1,\dotsc,\tau_l]^{[\mu]},\Y^n,h)\frac{h^k}{\sigma([(\tau^{(i_1)})^{m_1},\dotsc,(\tau^{(i_j)})^{m_j}]^{[\mu]})}\dF([\tau_1,\dotsc,\tau_l]^{[\mu]})(\y^n), \]
				since then any addend on the left-hand side of \eqref{eq:induction} is uniquely associated with the sum over the different root colors of a tree $\tau\in NT_k\setminus NT_{k-1}$. 
				Using \eqref{eq:delta(i)} and the definitions of $\sigma$, $\dF$ and $\widetilde\theta_\nu$ from \eqref{eq:sigmagamma}, \eqref{eq:elemdiff} and \eqref{eq:widetildetheta}, we indeed find
				\begin{equation*}
					\begin{aligned}
						\frac{h}{\hat\sigma(\hat I)}(\Fnu)^{(\lvert \hat I\rvert)}(\y^n)\bd^{\hat I}&=\frac{h}{\prod_{r=1}^jm_r!} \prod_{r=1}^j\left(\frac{(\theta(\tau^{(i_r)},\Y^n,h))^{m_r} h^{m_r\lvert \tau^{(i_r)}\rvert}}{(\sigma(\tau^{(i_r)}))^{m_r}} \right)(\Fnu)^{(l)}(\y^n)(\dF(\tau_1)(\y^n),\dotsc,\dF(\tau_l)(\y^n))  \\
						&= \widetilde \theta_\nu([(\tau^{(i_1)})^{m_1},\dotsc,(\tau^{(i_j)})^{m_j}]^{[\nu]},\Y^n,h)\frac{h^k}{\sigma([(\tau^{(i_1)})^{m_1},\dotsc,(\tau^{(i_j)})^{m_j}]^{[\nu]})}\dF([\tau_1,\dotsc,\tau_l]^{[\nu]})(\y^n)\\
						&=\sum_{\mu=1}^N\widetilde\theta_\nu([\tau_1,\dotsc,\tau_l]^{[\mu]},\Y^n,h)\frac{h^k}{\sigma([(\tau^{(i_1)})^{m_1},\dotsc,(\tau^{(i_j)})^{m_j}]^{[\mu]})}\dF([\tau_1,\dotsc,\tau_l]^{[\mu]})(\y^n)
					\end{aligned}
				\end{equation*}
				finishing the proof.
			\end{proof}
			With Lemma \ref{lem:hFl} we can prove the following result, which is the analogue to Lemma 313A in \cite{B16}. 
			\begin{lemma}\label{lem:hfnu}
				Define $d_i$ and $g_i^{[\nu]}$ for $i=1,\dotsc,s$ and $\nu=1,\dotsc, N$ as in \eqref{eq:pertcond}. Furthermore, let $k\in \N$ and $ \Fnu\in \mathcal C^{k+1}$ for $\nu=1,\dotsc, N$.  If
				\[\byi=\y^n+\sum_{\tau\in NT_{k-1}}\frac{1}{\sigma(\tau)}d_i(\tau,\Y^n,h)h^{\lvert \tau\rvert} \dF(\tau)(\y^n)+\O(h^k)\]
				then
				\[h\Fnu(\byi)=\sum_{\tau\in NT_{k}}\frac{1}{\sigma(\tau)}g^{[\nu]}_i(\tau,\Y^n,h)h^{\lvert \tau\rvert} \dF(\tau)(\y^n)+\O(h^{k+1}).\]
			\end{lemma}
			\begin{proof}
				The claim follows using Lemma \ref{lem:hFl} with $\theta(\tau,\Y^n,h)=d_i(\tau,\Y^n,h)$, which gives us
				\begin{equation*}
					\begin{aligned}
						\widetilde\theta_\nu (\tau,\Y^n,h)&=\begin{cases}
							\delta_{\nu\mu},& \tau=\rt[]^{[\mu]},\\
							\delta_{\nu\mu}\prod_{i=1}^ld_i(\tau_i,\Y^n,h),& \tau=[\tau_1,\dotsc,\tau_l]^{[\mu]}
						\end{cases} \\
						&=g_i^{[\nu]}(\tau,\Y^n,h).
					\end{aligned}
				\end{equation*}
				\vspace{-\baselineskip}
			\end{proof}
			\section{Results for Reducing Order Conditions of NSARK methods}
			As final intermediate results, we prove the following lemmas which are helpful to reduce the conditions for 3rd and 4th order MPRK and GeCo methods. Both families of schemes can be written in the form of an NSARK method with 
			\begin{equation*}
				a_{ij}^{[\nu]}(\Y^n,h)=a_{ij}\gamma_\nu^{(i)},\quad b_i^{[\mu]}(\Y^n,h)=b_i\gamma_\mu^{n+1}
			\end{equation*}
			for suitable solution-dependent functions $\gamma_\mu^{n+1}$ and $\gamma_\nu^{(i)}$. In the following we use these general functions to reduce the order conditions \eqref{eq:condp=3} and \eqref{eq:condp=4} for 3rd and 4th order, respectively. As we assume for Theorem~\ref{thm:main} that $a_{ij}^{[\nu]}(\Y^n,h)=\O(1)$ as $h\to 0$, it suffices to prove the following results.
			\begin{lemma}\label{lem:equivalent}
				Let $\A,\b,\c$ be the coefficients of an explicit 3-stage RK scheme of order 3, and let $\gamma_\nu^{(i)}=\O(1)$ as $h\to 0$. Then the conditions
				\begin{subequations}\label{eq:MPRKcondp=3}
					\begin{align}
						\gamma^{n+1}_\mu&=1 +\O(h^3), &\mu&=1,\dotsc,N,\\
						\sum_{i=2}^3 b_ic_i\gamma^{(i)}_\nu&=\frac12 +\O(h^2), &\nu&=1,\dotsc,N,\label{eq:cond3b}\\
						\sum_{i=2}^3 b_ic_i^2\gamma^{(i)}_\nu\gamma^{(i)}_\xi&=\frac13 +\O(h), &\nu,\xi&=1,\dotsc,N,\label{eq:cond3c}\\
						\sum_{i,j=2}^3 b_i a_{ij}c_j\gamma^{(i)}_\nu\gamma^{(j)}_\xi&=\frac16 +\O(h), &\nu,\xi&=1,\dotsc,N\label{eq:cond3d}
					\end{align}
				\end{subequations}
				and
				\begin{equation}\label{eq:MPRKcondp'=3}
					\begin{aligned}
						\gamma^{n+1}_\mu&=1 +\O(h^3), &\mu&=1,\dotsc,N,\\
						\sum_{i=2}^3 b_ic_i\gamma^{(i)}_\nu&=\frac12 +\O(h^2), &\nu&=1,\dotsc,N,\\
						\gamma^{(i)}_\nu&=1+\O(h),&\nu&=1,\dotsc,N,\quad i=2,3
					\end{aligned}
				\end{equation} 
				are equivalent for any solution and step-size dependent values of $\gamma^{n+1}_\mu$ and $\gamma^{(i)}_\nu$ for $i=2,3$ and $\mu,\nu=1,\dotsc,N$.
			\end{lemma}
			\begin{proof}
				It is easy to see that the conditions \eqref{eq:MPRKcondp=3} are fulfilled by any solution of \eqref{eq:MPRKcondp'=3}.
				To see that any solution of \eqref{eq:MPRKcondp=3} must satisfy \eqref{eq:MPRKcondp'=3}, consider the conditions from \eqref{eq:MPRKcondp=3} as $h\to 0$. From $\gamma_\nu^{(i)}=\O(1)$, any accumulation point of $\gamma^{(i)}_\nu$ is neither $\infty$ nor $-\infty$. In the following, we denote by $\Gamma^{(i)}_\nu$ an arbitrary accumulation point of  $\gamma^{(i)}_\nu$ as $h\to 0$. Moreover, since the underlying RK scheme is explicit with three stages, the only addend remaining on the left-hand side of \eqref{eq:cond3d} is 
				$b_3a_{32}c_2\gamma^{(3)}_\nu\gamma^{(2)}_\xi=\frac16\gamma^{(3)}_\nu\gamma^{(2)}_\xi.$
				Hence, for any accumulation point $\Gamma^{(i)}_\nu$, the conditions \eqref{eq:cond3b}, \eqref{eq:cond3c} with $\nu=\xi$, and \eqref{eq:cond3d} together with $c_1=0$ imply
				\begin{equation*}
					\begin{aligned}
						\sum_{i=2}^3 b_ic_i\Gamma^{(i)}_\nu&=\frac12, &\nu&=1,\dotsc,N,\\
						\sum_{i=2}^3 b_ic_i^2(\Gamma^{(i)}_\nu)^2&=\frac13, &\nu&=1,\dotsc,N,\\
						\Gamma^{(3)}_\nu\Gamma^{(2)}_\xi&=1, &\nu,\xi&=1,\dotsc,N.
					\end{aligned}
				\end{equation*}
				This system of equations possesses for any pair $(\nu,\xi)$ the unique solution $\Gamma^{(2)}_\xi=1$ and $\Gamma^{(3)}_\nu=1$ for all $\xi,\nu=1,\dotsc,N$, see \cite[Lemma 7]{MPRK3}. Finally, \eqref{eq:cond3c} with $\nu=\xi$  thus implies that $\gamma^{(i)}_\nu=1+\O(h)$ proving that \eqref{eq:MPRKcondp=3} and \eqref{eq:MPRKcondp'=3} are equivalent.
			\end{proof}
			\begin{lemma}\label{lem:equivalent4}
				Let $\A,\b,\c$ be the coefficients of an explicit 4-stage RK scheme of order 4, and let $\gamma_\nu^{(i)}=\O(1)$ as $h\to 0$. Then the conditions
				\begin{subequations}\label{eq:MPRKcondp=4}
					\begin{align}
						\gamma^{n+1}_\mu&=1 +\O(h^4), &\mu&=1,\dotsc,N,\\
						\sum_{i=2}^4 b_ic_i\gamma^{(i)}_\nu&=\frac12 +\O(h^3), &\nu&=1,\dotsc,N,\label{eq:cond4b}\\
						\sum_{i=2}^4 b_ic_i^2\gamma^{(i)}_\nu\gamma^{(i)}_\xi&=\frac13 +\O(h^2), &\nu,\xi&=1,\dotsc,N,\label{eq:cond4c}\\
						\sum_{i,j=2}^4 b_i a_{ij}c_j\gamma^{(i)}_\nu\gamma^{(j)}_\xi&=\frac16 +\O(h^2), &\nu,\xi&=1,\dotsc,N,\label{eq:cond4d}\\
						\sum_{i,j=2}^4 b_ic_i a_{ij}c_j\gamma^{(i)}_\nu\gamma^{(i)}_\xi\gamma^{(j)}_\eta&=\frac18 +\O(h), &\nu,\xi,\eta&=1,\dotsc,N,\label{eq:cond4e}\\
						\sum_{i=2}^4 b_ic_i^3 \gamma^{(i)}_\nu\gamma^{(i)}_\xi\gamma^{(i)}_\eta&=\frac14 +\O(h), &\nu,\xi,\eta&=1,\dotsc,N,\label{eq:cond4f}\\
						\sum_{i,j,k=2}^4 b_ia_{ij}a_{jk}c_k\gamma^{(i)}_\nu\gamma^{(j)}_\xi\gamma^{(k)}_\eta&=\frac{1}{4!}+\O(h), &\nu,\xi,\eta&=1,\dotsc,N,\label{eq:cond4g}\\
						\sum_{i,j=2}^4 b_i a_{ij}c_j^2\gamma^{(i)}_\nu\gamma^{(j)}_\xi\gamma^{(j)}_\eta&=\frac{1}{12} +\O(h), &\nu,\xi,\eta&=1,\dotsc,N\label{eq:cond4h}
					\end{align}
				\end{subequations}
				and
				\begin{equation}\label{eq:MPRKcondp'=4}
					\begin{aligned}
						\gamma^{n+1}_\mu&=1 +\O(h^4), &\mu&=1,\dotsc,N,\\
						\sum_{i=2}^4 b_ic_i\gamma^{(i)}_\nu&=\frac12 +\O(h^3), &\nu&=1,\dotsc,N,\\
						\gamma^{(i)}_\nu&=1+\O(h^2),&\nu&=1,\dotsc,N,\quad i=2,3´,4
					\end{aligned}
				\end{equation} 
				are equivalent for any solution and step-size dependent values of $\gamma^{n+1}_\mu$ and $\gamma^{(i)}_\nu$ for $i=2,3,4$ and $\mu,\nu=1,\dotsc,N$.
			\end{lemma}
			\begin{proof} This proof works analogously to the one of Lemma \ref{lem:equivalent}, where we need to come up with a substitute for \cite[Lemma 7]{MPRK3}.
				
				We first note that the conditions \eqref{eq:MPRKcondp=4} are fulfilled by any solution of \eqref{eq:MPRKcondp'=4}. 
				To see that any solution of \eqref{eq:MPRKcondp=4} must satisfy \eqref{eq:MPRKcondp'=4}, consider the conditions from \eqref{eq:MPRKcondp=4} as $h\to 0$. From $\gamma^{(i)}_\nu=\O(1)$, any accumulation point of $\gamma^{(i)}_\nu$ is neither $\infty$ nor $-\infty$. In the following, we denote by $\Gamma^{(i)}_\nu$ an arbitrary accumulation point of  $\gamma^{(i)}_\nu$ as $h\to 0$. Moreover, since the underlying RK scheme is explicit with four stages, the only addend remaining on the left-hand side of \eqref{eq:cond4g} is 
				\[b_4a_{43}c_3\gamma^{(4)}_\nu\gamma^{(3)}_\xi\gamma^{(2)}_\eta=\frac{1}{4!}\gamma^{(4)}_\nu\gamma^{(3)}_\xi\gamma^{(2)}_\eta.\]
				Hence, for any accumulation point $\Gamma^{(i)}_\nu$, the conditions \eqref{eq:MPRKcondp=4} together with $\nu=\xi=\eta$ and the order conditions for the underlying RK method imply 
				\begin{equation}\label{eq:redPGS_cond4}
					\begin{aligned}
						\sum_{i=2}^4 b_ic_i\left(\Gamma^{(i)}_\nu-1\right)&=0, &\nu&=1,\dotsc,N,\\
						\sum_{i=2}^4 b_ic_i^2\left((\Gamma^{(i)}_\nu)^2-1\right)&=0, &\nu&=1,\dotsc,N,\\
						\sum_{i,j=2}^4 b_i a_{ij}c_j\left(\Gamma^{(i)}_\nu\Gamma^{(j)}_\nu-1\right)&=0, &\nu&=1,\dotsc,N,\\
						\sum_{i,j=2}^4 b_ic_i a_{ij}c_j\left((\Gamma^{(i)}_\nu)^2\Gamma^{(j)}_\nu-1\right)&=0, &\nu&=1,\dotsc,N,\\
						\sum_{i=2}^4 b_ic_i^3 \left((\Gamma^{(i)}_\nu)^3-1\right)&=0, &\nu&=1,\dotsc,N,\\
						\Gamma^{(4)}_\nu\Gamma^{(3)}_\nu\Gamma^{(2)}_\nu-1&=0, &\nu&=1,\dotsc,N,\\
						\sum_{i,j=2}^4 b_i a_{ij}c_j^2\left(\Gamma^{(i)}_\nu(\Gamma^{(j)}_\nu)^2-1\right)&=0, &\nu&=1,\dotsc,N.
					\end{aligned}
				\end{equation}
				In what follows we fix $\nu\in \{1,\dotsc,N\}$. Then, we compute a reduced Gröbner basis\footnote{We refer to our Maple repository \cite{GBrepository} for the computation of the Gröbner bases for this work.} of the corresponding polynomial ideal generated by the polynomials on the left-hand sides of \eqref{eq:redPGS_cond4} in the ring $\R[\Gamma^{(2)}_\nu,\Gamma^{(3)}_\nu,\Gamma^{(4)}_\nu]$, yielding $\{\Gamma^{(2)}_\nu-1,\Gamma^{(3)}_\nu-1,\Gamma^{(4)}_\nu-1\}$. Hence, $\Gamma^{(i)}_\nu=1$ for $\nu=1,\dotsc,N$ and $i=2,3,4$ is the unique solution to \eqref{eq:redPGS_cond4}. As a result, \eqref{eq:cond4f} with $\nu=\xi=\eta$ implies that $\gamma^{(i)}_\nu=1+\O(h)$. This already allows us to neglect the conditions \eqref{eq:cond4e} to \eqref{eq:cond4h} in the following as they are now fulfilled by  $\gamma^{(i)}_\nu=1+\O(h)$. Substituting the ansatz\footnote{Formally, $x^{(i)}_\nu$ is an arbitrary accumulation point of $\tfrac{\gamma^{(i)}_\nu-1}{h}$ as $h\to 0$, which due to $\gamma^{(i)}_\nu=1+\O(h)$ cannot be $\pm \infty$. However, for the sake of simplicity, we refrain to introduce several $\gamma^{(i)}_\nu$ for every occurring accumulation point.} $\gamma^{(i)}_\nu=1+x^{(i)}_\nu h +\O(h^2)$ into the remaining conditions \eqref{eq:cond4b} to \eqref{eq:cond4d}, the resulting coefficients of $h$ must vanish, that is
				\begin{equation}\label{eq:redPGS_cond4b}
					\begin{aligned}
						\sum_{i=2}^4 b_ic_ix^{(i)}_\nu&=0,\\
						\sum_{i=2}^4 b_ic_i^22x^{(i)}_\nu&=0, \\
						\sum_{i,j=2}^4 b_i a_{ij}c_j(x^{(i)}_\nu+x^{(j)}_\nu)&=0.
					\end{aligned}
				\end{equation}
				We again compute a reduced Gröbner basis of the ideal generated by the left-hand side polynomials from \eqref{eq:redPGS_cond4b} in the polynomial ring $\R[x^{(2)}_\nu,x^{(3)}_\nu,x^{(4)}_\nu]$. The resulting Gröbner basis reads $\{x^{(2)}_\nu,x^{(3)}_\nu,x^{(4)}_\nu\}$ proving that the unique solution to the above polynomial system is given by $x^{(i)}_\nu=0$ for $\nu=1,\dotsc,N$ and $i=2,3,4$. With that we have demonstrated that $\gamma^{(i)}_\nu=1 +\O(h^2)$ which finishes the proof.
			\end{proof}

			\bibliographystyle{plain}
			\bibliography{refs}
		\end{document}